\documentclass[12pt,twoside,english,reqno,a4paper]{amsart}

\usepackage{listings,graphicx,amsmath,varioref,amscd,amssymb,color,bm,stmaryrd,amsthm,amsfonts,graphics,lineno}
\usepackage[english]{babel}
\newlength{\defbaselineskip}
\theoremstyle{plain}
\newtheorem{defin}{Definition}[section]
\newtheorem{theorem}[defin]{Theorem}

\newtheorem{lemma}[defin]{Lemma}

\newtheorem{remark}[defin]{Remark}
\numberwithin{equation}{section}

\def\dis{\displaystyle}
\def\rn{\mathbb{R}^N}
\def\RR{\mathbb{R}}
\def\NN{\mathbb{N}}
\def\io{\int_\Omega}
\def\supp{\text{\text{supp}}}
\def\sob{W^{1,p}_{0}(\Omega)}
\newcommand{\h}{H^{^*}}

\author[L. M. De Cave]{Linda Maria De Cave}

\author[R. Durastanti]{Riccardo Durastanti}

\author[F. Oliva]{Francescantonio Oliva}

\address[L. M. De Cave]{Institut f\"ur Mathematik, Universit\"at Z\"urich, Winterthurerstrasse 190, 8057 Z\"urich, Switzerland
\\linda.decave@math.uzh.ch}

\address[R. Durastanti]{Dipartimento di Scienze di Base e Applicate per l' Ingegneria, ``Sapienza" Universit\`a di Roma, Via Scarpa 16, 00161 Roma, Italy 
\\ riccardo.durastanti@sbai.uniroma1.it}

\address[F. Oliva]{Dipartimento di Scienze di Base e Applicate per l' Ingegneria, ``Sapienza" Universit\`a di Roma, Via Scarpa 16, 00161 Roma, Italy 
\\ francesco.oliva@sbai.uniroma1.it}


\begin{document}

\title[Possibly singular nonlinear elliptic equation]{Existence and uniqueness results for possibly singular nonlinear elliptic equations with measure data}

\maketitle 

\begin{abstract}
We study existence and uniqueness of solutions to a nonlinear elliptic boundary value problem with a general, and possibly singular, lower order term, whose model is
$$\begin{cases}
\dis -\Delta_p u  = H(u)\mu & \text{in}\ \Omega,\\
u>0 &\text{in}\ \Omega,\\
u=0 &\text{on}\ \partial\Omega.
\end{cases}$$
Here $\Omega$ is an open bounded subset of $\rn$ ($N\ge2$), $\Delta_p u:= \operatorname{div}(|\nabla u|^{p-2}\nabla u)$ ($1<p<N$) is the $p$-laplacian operator, $\mu$ is a nonnegative bounded Radon measure on $\Omega$ and $H(s)$ is a continuous, positive and finite function outside the origin which grows at most as $s^{-\gamma}$, with $\gamma\ge0$, near zero. 
\vskip 0.5\baselineskip
\end{abstract}

\tableofcontents

\section{Introduction}

We start recalling some literature concerning singular elliptic problems whose simplest model is given by
\begin{equation}
\begin{cases}
\displaystyle -\Delta u = \frac{\mu}{u^\gamma} &  \text{in}\, \Omega, \\
u>0 & \text{in}\, \Omega,\\
u=0 & \text{on}\ \partial \Omega,
\label{pbintro}
\end{cases}
\end{equation}
where $\Omega$ is an open bounded subset of $\mathbb{R}^N$ ($N \ge 2$), $\mu$ is a nonnegative datum and $\gamma> 0$. 
\\ \\ The pioneering papers concerning problem \eqref{pbintro} are \cite{crt}, \cite{lm} and  \cite{s}. 
\\In these works the authors consider the case of a smooth datum $\mu$, proving the existence of a unique classical solution $u\in C^2(\Omega)\cap C(\overline{\Omega})$ to \eqref{pbintro}. This solution does not belong to $C^2(\overline{\Omega})$ and, in \cite{lm}, it is proved that $u\in H^1_0(\Omega)$ if and only if  $\gamma<3$ and that, if $\gamma > 1$, the solution does not belong to $C^1(\overline{\Omega})$. For further informations on the H\"older continuity properties of the solution to \eqref{pbintro} see \cite{ghl}. 
\\ \\As concerns data $\mu$ merely in $L^1(\Omega)$, we mainly refer to \cite{bo}, where the authors prove the existence of a distributional solution to the problem working by approximation, desingularizing the right hand side of the equation. This solution belongs to $H^1_0(\Omega)$ if $\gamma= 1$, it is only in $H^1_{loc}(\Omega)$ if  $\gamma> 1$ and, finally, if $\gamma< 1$, it belongs to an homogeneous Sobolev space larger than $H^1_0(\Omega)$. In the case of measure data, we refer to \cite{do}, where the existence of a distributional solution is proved in the more general case of a quasilinear elliptic operator with quadratic coercivity and of a singular lower order term not necessarily non-increasing.  
\\ \\ As one can expect, uniqueness of solutions to \eqref{pbintro} is a challenging issue. 
\\If a solution to \eqref{pbintro} belongs to $H^1_0(\Omega)$, uniqueness holds (see \cite{boca}). In \cite{sz}, one can find a necessary and sufficient condition in order to have $H^1_0(\Omega)$ solutions to \eqref{pbintro} if $\gamma>1$ and $\mu\in L^1(\Omega)$ positive. If $\mu$ is a nonnegative function in $L^{\frac{2N}{N+2}}(\Omega)$ and the singular term is non-increasing, the solution to \eqref{pbintro}, defined through a transposition argument, is proved to be unique even if it belongs only to $H^1_{loc}(\Omega)$ (see \cite{gmm,gmm2}). If $\gamma>1$ and the datum is a diffuse measure, in \cite{po} the authors prove a uniqueness result. Finally, if $\Omega$ has a sufficiently regular boundary, uniqueness of solutions belonging only to $W^{1,1}_{loc}(\Omega)$ is proved by means of a suitable Kato's type argument when $\mu$ is a general measure and $H$ is a general non-increasing nonlinearity (see \cite{op2}).
\\ \\ Here we will study the following problem with a nonlinear principal operator
\begin{equation}
\begin{cases}
\displaystyle -\Delta_p u = H(u)\mu &  \text{in}\, \Omega, \\
u>0 & \text{in}\, \Omega,\\
u=0 & \text{on}\ \partial \Omega,
\label{pbintro2}
\end{cases}
\end{equation}
where, for $1<p<N$, $\Delta_p u:= \operatorname{div}(|\nabla u|^{p-2}\nabla u)$ is the $p$-laplacian operator, $\mu$ is a nonnegative bounded Radon measure on $\Omega$ and $H(s)$ is a nonnegative, continuous and finite function outside the origin, which, roughly speaking, behaves as $s^{-\gamma}$ ($\gamma\ge 0$) near zero. 
\\ \\In presence of a nonlinear principal operator the literature is more limited. We refer to \cite{dc} for the existence of a distributional solution when $H(s)=s^{-\gamma}$ and $\mu\in L^1(\Omega)$ while, in case of a general singular nonlinearity $H$ and $\mu\in L^{(p^*)'}(\Omega)$, we mention \cite{dcgop}. Furthermore, in \cite{cst}, the uniqueness of solutions which belong to $W^{1,p}_{loc}(\Omega)$ is proved if $\mu\in L^1(\Omega)$. This uniqueness result holds true in full generality in case of a star-shaped domain, while some more regularity on  f is needed if $\gamma>1$, $1<p\leq N$ and the domain is more general. Besides uniqueness of solutions belonging to $W^{1,p}_0(\Omega)$, which continues to hold even in presence of a nonlinear operator, many of the techniques used to prove uniqueness in the linear case $p=2$ can not be extended to the general case $p>1$. 
\\ \\We stress that uniqueness for solutions to \eqref{pbintro2} is an hard issue even if $H\equiv 1$. Indeed, in general, having a distributional solutions is not sufficient to deduce uniqueness which holds in the framework of the so-called \textit{renormalized solution} (see Definition \ref{renormalized} below, given in the case of a general $H$). The notion of renormalized solution formally selects a particular solution among the distributional ones. We also highlight that the existence of a renormalized solution for a continuous and finite function $H$ is given in \cite{mupo} when $p=2$; this solution is also unique if $H$ is non-increasing and $\mu$ is diffuse with respect to the harmonic capacity (see Section \ref{sec:not} below). We refer the interested reader to \cite{dmop} for a complete account on the renormalized framework for problems whose model is given by \eqref{pbintro2} with $H\equiv1$ and the positivity requirement on $u$ is removed ($\mu$ is not necessarily nonnegative).      
\\ \\Without the aim to be complete, we refer to various works treating different aspects of problems as in \eqref{pbintro} and in \eqref{pbintro2}. The literature concerning the case of linear operators is \cite{af,af2,am,bgh,c,cd,car,cc,cc2,diaz,edr}. For more general operators we refer to \cite{dc,do,gcs,kl,op}. Finally, also symmetry of solutions is considered in  \cite{cgs,cms,tromb}.
\\ \\ Here we show the existence of a distributional solution $u$ to \eqref{pbintro2} despite a nonlinear operator, a measure as datum and a general lower order term. 
\\The most interesting fact is that $u$ turns out also to be a renormalized solution to the singular problem if $\gamma\le 1$. This is strictly related to the fact that, in this case, the truncations of $u$ at any level $k$, $T_k(u)$, belong to the space of finite energy, differently to the case $\gamma>1$, where $T_k(u)$ is, in general, only in $W^{1,p}_{loc}(\Omega)$. 
\\As already stressed, the existence of a renormalized solution is linked to the uniqueness of the solution to \eqref{pbintro2}. Indeed, in case of a diffuse measure datum and of a non-increasing $H$, without requiring any additional assumption on $\Omega$ and on $\mu$, we are able to prove that the renormalized solution is unique even in presence of a principal operator which can be way more general than the $p$-laplacian. 
\\It is worth noting that, at the best of our knowledge, our result is new even in case of a continuous and finite nonlinearity $H$ (i.e., if $\gamma=0$), so that we are also providing an extension of the results of \cite{mupo} to the case $p\neq2$.
\\ \\We give a brief plan of the paper. Section \ref{sec:not} is devoted to present the preliminary results and the notations used throughout the paper. In Section \ref{sec:set} we provide the assumptions, the notions of solutions we are adopting and the statements of the existence and uniqueness theorems. In Section \ref{exgamma=0} we prove the existence theorem when $H$ is finite. In Section \ref{sec:maggiore0} we provide the approximation scheme and the main tools in preparation of the proof of the theorems when $H$ can blow up at the origin. In Section \ref{mainresults} we apply all tools of the previous section to deduce the existence and uniqueness theorems in their full generality. Finally, in Section \ref{H=0}, we provide some further results concerning the regularity of a solution to \eqref{pbintro2} when $H(s)$ can degenerate (i.e., becomes zero) at some point $s>0$.  

\section{Notations and preliminaries}\label{sec:not}

We denote by $C_b(\RR)$ the space of continuous and bounded functions on $\RR$ and by $C_c(\Omega)$ the space of continuous functions with compact support in $\Omega$; the latter one will be an open bounded subset of $\RR^N$ ($N\ge 2$) in the entire paper.
If no otherwise specified, we will denote by $C$ several constants whose value may change from line to line and, sometimes, on the same line. These values will only depend on the data (for instance $C$ may depend on $\Omega$, $\gamma$, $N$) but they will never depend on the indexes of the sequences we will introduce. Moreover, in order to take into account the order of the limits, we will denote by $\epsilon(n,r,\nu)$ any quantity such that
$$\limsup_{\nu\rightarrow0}\,\limsup_{r\rightarrow\infty}\,\limsup_{n\rightarrow\infty}\epsilon(n,r,\nu)=0.$$
For a fixed $k>0$, we introduce the truncation functions $T_{k}$ and $G_{k}$
$$T_k(s)=\max (-k,\min (s,k)), $$
$$G_k(s)=(|s|-k)^+ \operatorname{sign}(s),$$
and we also define the functions $\pi_k:\RR\rightarrow\RR$ and $\theta_k:\RR\rightarrow\RR$
\begin{equation}\label{pi}\pi_k(s)=\frac{T_k(s-T_k(s))}{k},\end{equation}
\begin{equation}\label{teta}\theta_k(s)=1-|\pi_k(s)|.\end{equation}
From now onwards, when employing functions denoted by $\pi_k$ or $\theta_k$, we will mean the previous functions.
\\Let $f\in L^1_{loc}(\Omega)$ then $x\in\Omega$ is a Lebesgue point of $f$ if there exists $\widehat{f}(x)\in\RR$ such that 
$$\lim_{\rho\rightarrow0}\frac{1}{|B_{\rho}(x)|}\int_{B_{\rho}(x)}\big|f-\widehat{f}(x)\big|=0.$$ 
By the Lebesgue differentiation Theorem, almost every point $x\in\Omega$ is a Lebesgue point of $f$ and $f(x)=\widehat{f}(x)$. We denote as $\mathcal{L}_{f}$ the set of Lebesgue points of a function $f\in L^1_{loc}(\Omega)$. 
\\ The standard $p$-capacity of a Borel set $E\subset\Omega$ is defined by
$$\operatorname{cap}(E,\Omega)= \inf \int_{\Omega}|\nabla u|^p \ \ \text{with $u \in W^{1,p}_0(\Omega)$ : $u \ge 1$\;\text{a.e. in a neighborhood of}\;E}.$$
A function $u$ is said to be cap$_{p}$-quasi continuous if for every $\epsilon>0$ there exists an open set $E\subset\Omega$ such that $\operatorname{cap}(E)<\epsilon$ and $\left.u\right|_{\Omega\setminus E}$ is continuous in $\Omega\setminus E$.
\\Moreover for every $u \in W^{1,p}(\Omega)$ there exists a cap$_p$-quasi continuous representative $\tilde{u}$ yielding $u=\tilde{u}$ almost everywhere in $\Omega$ and if $\widehat{u}$ is another cap$_p$-quasi continuous representative of $u$, then $\widehat{u}=\tilde{u}$ cap$_p$ almost everywhere in $\Omega$. We will always refer to the cap$_p$-quasi continuous representative when dealing with functions in $W^{1,p}(\Omega)$. 
\\We denote the space of bounded Radon measures by $\mathcal{M}(\Omega)$. Let us a recall that $\mu\in\mathcal{M}(\Omega)$ is said to be {\itshape diffuse} with respect to the $p$-capacity if for every Borel set $B\subset \Omega$ such that $\operatorname{cap}_p(B)=0$ it results $\mu(B)=0$. Moreover $\mu$ is said to be {\itshape concentrated} on a Borel set $B\subset \Omega$ if $\mu(E)=\mu(E\cap B)$ for every $E\subset \Omega$. 
\\ It follows from \cite{fuku} that every $\mu\in \mathcal{M}(\Omega)$ can be uniquely decomposed as
$$\mu=\mu_d+\mu_c,$$
where $\mu_d$ is diffuse and $\mu_c$ is concentrated on a set of zero $p$-capacity and that, if $\mu \ge 0$, then $\mu_d,\mu_c\ge 0$.
\\ Furthermore, in \cite{bgo}, is proved the following decomposition result
\begin{equation*}
\mu\in \mathcal{M}(\Omega) \text{   is diffuse if and only if }\mu=f-\operatorname{div}(F)\;\;\text{with}\;f\in L^1(\Omega),F\in L^{p'}(\Omega)^N.
\label{dec}
\end{equation*}
The latter decomposition is not unique since $L^1(\Omega)\cap W^{-1,p'}(\Omega)\neq\{0\}$. We recall that a sequence of measures $\mu_n$ converges to $\mu$ in the narrow topology of $\mathcal{M}(\Omega)$ if 
$$\lim_{n\to\infty} \io \varphi d\mu_n=\io \varphi d\mu  \quad \forall \varphi \in C_b(\Omega).$$
\\Here we collect some results contained in \cite{b6} and \cite{dmop}.
\begin{lemma}\label{dalmaso}
Let $\lambda\in\mathcal{M}(\Omega)$ be nonnegative and concentrated on a set $E$ such that $\operatorname{cap}_p(E)=0$. Then, for every $\nu>0$, there exists a compact subset $K_{\nu}\subset E$ and a function $\Psi_{\nu}\in C^{\infty}_c(\Omega)$ such that the following hold
$$\lambda(E\setminus K_{\nu})<\nu,\;\;0\le\Psi_{\nu}\le1\;\text{in}\;\Omega,\;\;\Psi_{\nu}\equiv1\;\text{in}\;K_{\nu},\,\;\lim_{\nu\rightarrow0}\|\Psi_{\nu}\|_{W^{1,p}_0(\Omega)}=0.$$
\end{lemma}
In the entire paper we will denote by $\Psi_\nu$ a function with the properties of the previous Lemma.
\begin{lemma}\label{dalmaso2}
Let $u:\Omega\to\RR$ be a measurable function almost everywhere finite on $\Omega$ such that $T_k(u)\in W^{1,p}_0(\Omega)$ for every $k>0$. Then there exists a measurable function $v:\Omega\to\RR^N$ such that
$$\nabla T_k(u)=v\chi_{\{|u|\leq k\}},$$ 
and we define the gradient of $u$ as $\nabla u=v$. Moreover, if 
$$\io |\nabla T_k(u)|^p\leq C(k+1)\quad\forall k>0,$$
then $u$ is cap$_p$-almost everywhere finite, i.e. $\operatorname{cap}_p\{x\in\Omega\,:\,|u(x)|=+\infty\}=0$, and there exists a cap$_p$-quasi continuous representative $\tilde{u}$ of $u$, namely a function $\tilde{u}$ such that $\tilde{u}=u$ almost everywhere in $\Omega$ and $\tilde{u}$ is cap$_p$-quasi continuous.
\end{lemma}
In what follows, when dealing with a function $u$ that satisfies the assumptions of the previous Lemma, we will always consider its cap$_p$-quasi continuous representative.
\begin{lemma}\label{dalmaso3}
Let $\mu_d$ be a nonnegative diffuse measure with respect to the $p$-capacity and let $u\in W^{1,p}_0(\Omega)\cap L^\infty(\Omega)$ be a nonnegative function. Then, up to the choice of its cap$_p$-quasi continuous representative, $u$ belongs to $L^\infty(\Omega,\mu_d)$ and
$$\io ud\mu_d\leq\|u\|_{L^\infty(\Omega)}\mu_d(\Omega).$$
\end{lemma}
We recall also the following very well known consequence of the Egorov Theorem.
\begin{lemma}\label{Orsina}
Let $f_n$ be a sequence converging to $f$ weakly in $L^1(\Omega)$ and let $g_n$ be a sequence converging to $g$ almost everywhere in $\Omega$ and *-weakly in $L^\infty(\Omega)$. Then
$$\dis \lim_{n\to +\infty} \int_{\Omega} f_n g_n = \int_\Omega fg.$$
\end{lemma}

\section{Main assumptions and results}\label{sec:set}

We will consider the following
\begin{equation}
\begin{cases}
\displaystyle -\operatorname{div}(a(x,\nabla u)) = H(u)\mu &  \text{in}\, \Omega, \\
u=0 & \text{on}\ \partial \Omega,
\label{pbmain}
\end{cases}
\end{equation}
where $\displaystyle{a(x,\xi):\Omega\times\mathbb{R}^{N} \to \mathbb{R}^{N}}$ is a Carath\'eodory function satisfying the classical Leray-Lions structure conditions for $1<p<N$, namely
\begin{align}
&a(x,\xi)\cdot\xi\ge \alpha|\xi|^{p}, \ \ \ \alpha>0,
\label{cara1}\\
&|a(x,\xi)|\le \beta|\xi|^{p-1}, \ \ \ \beta>0,
\label{cara2}\\
&(a(x,\xi) - a(x,\xi^{'} )) \cdot (\xi -\xi^{'}) > 0,
\label{cara3}	
\end{align}
for every $\xi\neq\xi^{'}$ in $\mathbb{R}^N$ and for almost every $x$ in $\Omega$.
\\Moreover $\mu$ is a nonnegative bounded Radon measure on $\Omega$ uniquely decomposed as the sum $\mu_d+\mu_c$, where $\mu_d$ is a diffuse measure with respect to the $p$-capacity and $\mu_c$ is a measure concentrated on a set of zero $p$-capacity. We underline that (see Remark \ref{linda} below) we will always assume 
\begin{equation}\label{hmu}
\mu_d\not\equiv0.
\end{equation}
Finally, if not otherwise specified, $H:(0,+\infty)\to (0,+\infty)$ is a continuous function, possibly blowing up at the origin, such that the following properties hold true
\begin{equation}
\exists\,\lim_{s\to \infty} H(s):=H(\infty)<\infty
\label{h}\end{equation}
\begin{equation}\label{h1}
\displaystyle \exists\ {C},s_0>0, \gamma\ge 0\;\ \text{s.t.}\;\  H(s)\le \frac{C}{s^\gamma} \ \ \text{if} \ \ s<s_0.
\end{equation} 
We emphasize that, since we are allowing $\gamma$ to be zero, we are taking into account also the case of a bounded $H$. Moreover the assumption on the strict positivity of $H$ is a technical one needed to handle the case in which the singular part of the measure is not identically zero, as widely explained in Section \ref{H=0}. 
\\ \\ \noindent First of all it is worth to clarify what we mean by \emph{solution} to problem \eqref{pbmain}. We provide two different notions of solution.
\begin{defin}\label{renormalized}
Let $a$ satisfy \eqref{cara1}, \eqref{cara2}, \eqref{cara3}, let $\mu$ be a nonnegative bounded Radon measure and let $H$ satisfy \eqref{h} and \eqref{h1}. A positive function $u$, which is almost everywhere finite on $\Omega$, is a \emph{renormalized solution} to problem \eqref{pbmain} if $T_k(u) \in W^{1,p}_0(\Omega)$ for every $k>0$ and if the following hold
\begin{gather}
H(u)S(u)\varphi \in L^1(\Omega,\mu_d)\;\text{and}\notag\\
\int_{\Omega} a(x,\nabla u)\cdot\nabla \varphi S(u) + \int_{\Omega}a(x,\nabla u)\cdot\nabla u S'(u)\varphi = \int_{\Omega}H(u)S(u)\varphi d\mu_d \label{ren1}\\
\forall S \in W^{1,\infty}(\mathbb{R})\;\text{with compact support and}\;\forall\varphi \in W^{1,p}_0(\Omega)\cap L^\infty(\Omega),\notag
\end{gather}
\begin{equation}
\label{ren2}
\lim_{t\to \infty} \frac{1}{t}\int_{\{t< u< 2t\}} a(x,\nabla u)\cdot \nabla u \varphi = H(\infty)\int_{\Omega}\varphi d\mu_c\quad\forall \varphi \in C_b(\Omega).
\end{equation}
\end{defin}
\begin{defin}\label{distributional}
Let $a$ satisfy \eqref{cara1}, \eqref{cara2}, \eqref{cara3}, let $\mu$ be a nonnegative bounded Radon measure and let $H$ satisfy \eqref{h} and \eqref{h1}. A positive and measurable function $u$ such that $|\nabla u|^{p-1} \in L^1_{loc}(\Omega)$ is a \emph{distributional solution} to problem \eqref{pbmain} if $  H(u) \in L^1_{loc}(\Omega,\mu_d)$, and the following hold
\begin{equation}\label{troncate}
T_k^\frac{\tau-1+p}{p}(u) \in W^{1,p}_0(\Omega)\quad\forall k>0,\quad\text{where}\quad\dis\tau= \max\left(1, \gamma\right),
\end{equation}
and
\begin{equation} \displaystyle \int_{\Omega}a(x,\nabla u) \cdot \nabla \varphi =\int_{\Omega} H(u)\varphi d\mu_d + H(\infty)\int_{\Omega} \varphi d\mu_c \ \ \ \forall \varphi \in C^1_c(\Omega).\label{distrdef}\end{equation}
\end{defin}
The notion of renormalized solution is way more general than the distributional one. Indeed, if $\gamma\leq1$, it results that the former implies the latter one.
\begin{lemma}
\label{equivrindis}
Let $\gamma\leq1$ and let $u$ be a renormalized solution to \eqref{pbmain}. Then $u$ is also a distributional solution to \eqref{pbmain}.
\end{lemma}
\begin{proof}
It follows from the definition of renormalized solution that \eqref{troncate} holds. Taking as test functions in \eqref{ren1} $S=\theta_t$, where $\theta_t$ is defined in \eqref{teta}, and $\varphi=T_k(u)$, with $s_0<k<t$, we obtain
\begin{equation*}
\io a(x,\nabla u)\cdot \nabla T_k(u)\theta_t(u)\leq\frac{k}{t}\int_{\{t<u<2t\}} a(x,\nabla u)\cdot \nabla u + \io H(u)T_k(u)\theta_t(u)d\mu_d.
\end{equation*}
Using \eqref{cara1} and \eqref{h1}, we find
\begin{align*}
\alpha\io |\nabla T_k(u)|^p &\leq \frac{k}{t}\int_{\{t< u< 2t\}} a(x,\nabla u)\cdot \nabla u +\int_{\{u<s_0\}}H(u)T_k(u)\theta_t(u)d\mu_d\\
&+\int_{\{u\geq s_0\}} H(u)T_k(u)\theta_t(u)d\mu_d \leq \frac{k}{t}\int_{\{t<u<2t\}} a(x,\nabla u)\cdot \nabla u\\
&+Cs_0^{1-\gamma}\|\mu_d\|_{\mathcal{M}(\Omega)}+k\|H\|_{L^{\infty}([s_0,+\infty))}\|\mu_d\|_{\mathcal{M}(\Omega)},
\end{align*}
so that, passing to the limit as $t\to\infty$, we find that there exists a constant $C>0$ such that
\begin{equation}
\label{rendis1}
\io |\nabla T_k(u)|^p\leq C(k+1), \quad \forall k>0.
\end{equation}
By \eqref{rendis1}, using Lemma \ref{dalmaso2} we deduce that $u$ is cap$_p$-almost everywhere finite and cap$_p$-quasi continuous and, using Lemma $4.2$ of \cite{b6}, we deduce moreover that $\displaystyle |\nabla u|^{p-1}\in L^1(\Omega)$. Now taking $\varphi\in C_c^1(\Omega)$ and $S=\theta_t$ in \eqref{ren1} we obtain
\begin{equation}
\label{rendis2}
\io a(x,\nabla u)\cdot \nabla\varphi\theta_t(u)=\frac{1}{t}\int_{\{t<u<2t\}} a(x,\nabla u)\cdot \nabla u\varphi + \io H(u)\varphi\theta_t(u)d\mu_d.
\end{equation}
By \eqref{ren1} it results $H(u)\theta_1(u)\varphi\in L^1{(\Omega,\mu_d)}$, and so, using Lemma \ref{dalmaso3}, we find
\begin{align*}
&\io |H(u)\varphi| d\mu_d=\int_{\{u<1\}}H(u)|\varphi| d\mu_d+\int_{\{u\geq 1\}}H(u)|\varphi|d\mu_d \\
&\leq \io H(u)\theta_1(u)|\varphi|d\mu_d+\|H\|_{L^{\infty}([1,+\infty))}\|\varphi\|_{L^{\infty}(\Omega)}\|\mu_d\|_{\mathcal{M}(\Omega)}\leq C,
\end{align*}
that implies $H(u)\in L^1_{loc}(\Omega,\mu_d)$. Letting $t$ go to infinity in \eqref{rendis2} we obtain, applying Lebesgue's Theorem for general measure and \eqref{ren2}, that \eqref{distrdef} holds. Hence $u$ is a distributional solution to \eqref{pbmain}.
\end{proof}
We will prove the following results.
\begin{theorem}\label{teoexrinuniqueness}
Let $a$ satisfy \eqref{cara1}, \eqref{cara2}, \eqref{cara3}, and let $\mu$ be a nonnegative bounded Radon measure which satisfies \eqref{hmu}. If $H$ satisfies \eqref{h} and \eqref{h1} with $\gamma\le 1$, there exists a renormalized solution $u$ to problem \eqref{pbmain}. Moreover,
\begin{itemize}
\item [i)] if $1<p\leq2-\frac{1}{N}$ then $u^{p-1}\in L^q(\Omega)$  $\forall\, q<\frac{N}{N-p}$ and $\quad|\nabla u|^{p-1} \in L^q(\Omega)$ $\forall\, q<\frac{N}{N-1}$,
\item [ii)] if $p>2-\frac{1}{N}$ then $u \in W_0^{1,q}(\Omega)\;\forall\,q<\frac{N(p-1)}{N-1}$.
\end{itemize}
Finally, if $H$ is non-increasing and $\mu_c\equiv 0$, then $u$ is unique.
\end{theorem}
\begin{theorem}\label{teoexistence}
Let $a$ satisfy \eqref{cara1}, \eqref{cara2}, \eqref{cara3}, and let $\mu$ be a nonnegative bounded Radon measure which satisfies \eqref{hmu}. If $H$ satisfies \eqref{h} and \eqref{h1}, there exists a distributional solution $u$ to problem \eqref{pbmain} such that 
$$u^{p-1}\in L^q_{loc}(\Omega)\;\forall\,q<\frac{N}{N-p}\quad\text{and}\quad|\nabla u|^{p-1} \in L^q_{loc}(\Omega)\;\forall\,q<\frac{N}{N-1}.$$
\end{theorem}
\begin{remark}
From Theorems \ref{teoexrinuniqueness}, \ref{teoexistence} and Lemma \ref{equivrindis}, we deduce that, for any nonlinearity $H$ satisfying \eqref{h} and \eqref{h1} with $\gamma\le 1$, we are able to find a renormalized solution that is also a distributional one. Otherwise, if $H$ blows up too fast at the origin (i.e. $\gamma>1$ in \eqref{h1}), the solution loses the weak trace in the classical Sobolev sense and we are only able to prove the existence of a distributional solution. We underline that the renormalized framework seems to be the natural one associated to this kind of problems, since it is well posed with respect to uniqueness, at least in case of a non-increasing nonlinearity $H$.  
\end{remark}
\begin{remark}\label{linda}
As concerns the assumption \eqref{hmu}, we underline that, if $H(0)<\infty$, we can prove the existence of a renormalized solution to \eqref{pbmain} even if it results $\mu_d\equiv0$, since we never use that $\mu_d\not\equiv0$ in the proof of Theorem \ref{teoexistence} (cf. Section \ref{exgamma=0}). If instead $H(0)=\infty$, then we do not to treat the case $\mu_d\equiv0$ to avoid nonexistence results (in the approximation sense) analogous to the ones of Section \ref{sec:maggiore0} of \cite{do}.
\\Furthermore, in case $\mu_d\equiv0$, our notions of solution formally lead us to the following problem with linear lower order term
$$\begin{cases}
\displaystyle -\operatorname{div}(a(x,\nabla u)) = H(\infty)\mu_c &  \text{in}\, \Omega, \\
u=0 & \text{on}\ \partial \Omega,
\end{cases}$$
which could be analyzed using classical tools.
\end{remark}

\section{Proof of existence in case of a finite $H$}\label{exgamma=0}

We start proving the existence of a renormalized solution in case of a finite nonlinearity $H$, namely assuming $\gamma=0$ in \eqref{h1}. \\
\\We introduce the following scheme of approximation
\begin{equation}\begin{cases}
\displaystyle -\operatorname{div}(a(x,\nabla u_n)) = H(u_n)\mu_n &  \text{in}\, \Omega, \\
u_n=0 & \text{on}\ \partial \Omega,
\label{pbapproxlimitata}
\end{cases}\end{equation}
where $\mu_{n}=\mu_{n,d} + \mu_{n,c} = f_n -\operatorname{div}(F_n) + \mu_{n,c}$. Following \cite{bgo} we suppose that:
\begin{align}\label{approxmisura}
&0\le f_n \in L^{\infty}(\Omega),\quad f_n  \rightarrow f \text{ weakly in } L^1(\Omega),\nonumber\\ 
&F_n \in W^{1,\infty}_0(\Omega)^N,\quad F_n \rightarrow F \text{ in } L^{p'}(\Omega)^N,\\ \nonumber
&0\le \mu_{n,c} \in L^{\infty}(\Omega),\quad\mu_{n,c}  \rightarrow \mu_c \text{ in the narrow topology of } \mathcal{M}(\Omega).		
\end{align}
Moreover it results that $\|\mu_n\|_{L^1(\Omega)}\le C$. 
\\Since $H$ is a continuous function satisfying \eqref{h} and \eqref{h1} with $\gamma=0$ and $a$ satisfies \eqref{cara1}, \eqref{cara2} and \eqref{cara3} with $1<p<N$, the existence of a weak solution $u_n\in W^{1,p}_0(\Omega)\cap L^\infty(\Omega)$ is guaranteed by \cite{ll}. Furthermore, since $H$ and $\mu_n$ are nonnegative functions, we also have that $u_n$ is nonnegative. Taking $S(u_n)\varphi$ as test function in the weak formulation of \eqref{pbapproxlimitata} where $S\in W^{1,\infty}(\mathbb{R})$ and has compact support and $\varphi \in W^{1,p}_0(\Omega)\cap L^\infty(\Omega)$ we obtain
\begin{gather}
\int_{\Omega} a(x,\nabla u_n)\cdot\nabla \varphi S(u_n) + \int_{\Omega}a(x,\nabla u_n)\cdot\nabla u_nS'(u_n)\varphi = \int_{\Omega}H(u_n)S(u_n)\varphi \mu_n. \label{ren_fin}
\end{gather}
Moreover, since $a(x,\nabla u_n)\cdot \nabla u_n \in L^1(\Omega)$, we deduce
\begin{equation}\label{r0}
\displaystyle \lim_{t\to \infty} \frac{1}{t} \int_{\{t<u_n<2t\}} a(x,\nabla u_n)\cdot \nabla u_n \varphi =0\quad\forall \varphi \in C_b(\Omega),
\end{equation} 
namely $u_n$ is also a renormalized solution to \eqref{pbapproxlimitata}. We need some a priori estimates on $u_n$.
\begin{lemma}\label{lemmastime}
Let $u_n$ be a solution to \eqref{pbapproxlimitata}. Then $T_k(u_n)$ is bounded in $W^{1,p}_0(\Omega)$ for every fixed $k>0$. Moreover:
\begin{itemize}
\item[i)] if $p>2-\frac{1}{N}$, $u_n$ is bounded in $W_0^{1,q}(\Omega)$ for every $q<\frac{N(p-1)}{N-1}$;
\item[ii)] if $1<p\le 2-\frac{1}{N}$, $u_n^{p-1}$ is bounded in $L^q(\Omega)$ for every $q<\frac{N}{N-p}$ and $|\nabla u_n|^{p-1}$ is bounded in $L^q(\Omega)$ for every $q<\frac{N}{N-1}$.	
\end{itemize}
Finally $u_n$ converges almost everywhere in $\Omega$ to a function $u$, which is cap$_p$-almost everywhere finite and cap$_p$-quasi continuous.
\end{lemma}
\begin{proof}
We take $T_{k}(u_n)$ in the weak formulation of \eqref{pbapproxlimitata} obtaining
$$\int_\Omega a(x,\nabla T_{k}(u_n))\cdot \nabla T_k(u_n) =\int_\Omega H(u_n)T_{k}(u_n) \mu_n.$$
Then, using (\ref{cara1}) and \eqref{approxmisura}, we find
\begin{equation}\label{Tkunif}\alpha\int_\Omega |\nabla T_{k}(u_n)|^p \leq k\|H\|_{L^{\infty}(\RR)}\|\mu_n\|_{L^1(\Omega)}\leq Ck,
\end{equation}
namely $T_{k}(u_n)$ is bounded in $W_0^{1,p}(\Omega)$ with respect to $n$. 
\\Then, if $p>2-\frac{1}{N}$, by the computations of Subsection II.4 in \cite{bg}, it follows that ${u_n}$ is bounded in $W_0^{1,q}(\Omega)$ for every $q<\frac{N(p-1)}{N-1}$. So there exists a nonnegative function $u$ belonging to $W_0^{1,q}(\Omega)$ for every $q<\frac{N(p-1)}{N-1}$ such that $u_n$ converges to $u$ almost everywhere in $\Omega$ and weakly in $W_0^{1,q}(\Omega)$ for every $q<\frac{N(p-1)}{N-1}$.
\\Otherwise, if $1<p\leq 2-\frac{1}{N}$, it results that $0<\frac{N(p-1)}{N-1}\leq1$ and we cannot proceed as before. Anyway, from \eqref{Tkunif}, using Lemma $4.1$ and Lemma $4.2$ of \cite{b6} we deduce that $u_n$ is bounded in the Marcinkiewicz space $M^{\frac{N(p-1)}{N-p}}(\Omega)$ and that $|\nabla u_n|$ is bounded in the Marcinkiewicz space $M^{\frac{N(p-1)}{N-1}}(\Omega)$. In particular $u_n^{p-1}$ is bounded in $L^q(\Omega)$ for every $q<\frac{N}{N-p}$ and $|\nabla u_n|^{p-1}$ is bounded in $L^q(\Omega)$ for every $q<\frac{N}{N-1}$. Furthermore, by \eqref{Tkunif}  we deduce that $T_k(u_n)$ is a Cauchy sequence in $L^p(\Omega)$ for all $k>0$, so that, up to subsequences, it is a Cauchy sequence in measure for each $k>0$. Then, using the Marcinkiewicz estimates on $u_n$, we find that $u_n$ is a Cauchy sequence in measure. To prove this property we begin by observing that for all $k,\sigma>0$ and for all $n,m\in\NN$, it results that
\begin{equation}\label{divido}
\{|u_n-u_m|>\sigma\}\subseteq\{|u_n|\geq k\}\cup\{|u_m|\geq k\}\cup\{|T_k(u_n)-T_k(u_m)|>\sigma\}.
\end{equation}
Now, if $\varepsilon>0$ is fixed, the Marcinkiewicz estimates imply that there exists a $\overline{k}>0$ such that 
$$\left|\{|u_n|> k\}\right|<\frac{\varepsilon}{3},\;\;\left|\{|u_m|> k\}\right|<\frac{\varepsilon}{3}\;\forall n,m\in\NN,\;\forall k>\overline{k},$$
while, using that $T_{k}(u_{n})$ is a Cauchy sequence in measure for each $k>0$ fixed, we deduce that there exists $\eta_{\varepsilon}>0$ such that
$$\left|\{|T_k(u_n)-T_k(u_m)\right|>\sigma\}|<\frac{\varepsilon}{3}\;\forall n,m>\eta_{\varepsilon},\;\forall\sigma>0.$$
Thus, if $k>\overline{k}$, from \eqref{divido} we obtain that 
$$\left|\{|u_n-u_m|>\sigma\}\right|<\varepsilon\quad\forall n,m\geq\eta_{\varepsilon},\;\forall\sigma>0,$$ 
and so that $u_n$ is a Cauchy sequence in measure. Then, in case $1<p\leq 2-\frac{1}{N}$, there exists a nonnegative measurable function $u:\Omega\to\RR$  to which $u_n$ converges almost everywhere in $\Omega$. Since $u_n^{p-1}$ is bounded in $L^q(\Omega)$ for every $q<\frac{N}{N-p}$, thanks to the almost everywhere convergence and Vitali's Theorem, we find that $u^{p-1}\in L^q(\Omega)$ for every $q<\frac{N}{N-p}$. This implies that the limit function $u$ is almost everywhere finite. 
\\Hence, in all cases, it results
\begin{equation}\label{weakconv}
T_k(u_n) \rightarrow T_k(u) \text{ weakly in } W^{1,p}_0(\Omega) \text{ for every } k>0 \text{ and a.e.} \text{ in } \Omega.
\end{equation}
\\Finally, thanks to \eqref{Tkunif}, by weak lower semicontinuity we deduce
\begin{equation*}\label{Tk}
\io |\nabla T_k(u)|^p\leq C(k+1)\quad\forall k>0,
\end{equation*}
and so, by the previous and Lemma \ref{dalmaso2}, we conclude that the function $u$ is cap$_p$-almost everywhere finite and cap$_p$-quasi continuous. 
\end{proof}
The previous lemma guarantees only the weak convergence of $T_k(u_n)$ towards $T_k(u)$ in $W^{1,p}_0(\Omega)$. In the next lemma we prove the strong convergence of truncations in $W^{1,p}_0(\Omega)$, which, in turn, will assure the almost everywhere convergence of $\nabla u_n$ to $\nabla u$ in $\Omega$.
\begin{lemma}\label{stronghlim}
Let $u_n$ be a solution to \eqref{pbapproxlimitata}. Then $T_k(u_n)$ converges to $T_k(u)$ in $W^{1,p}_0(\Omega)$ for every fixed $k>0$.
\end{lemma}
\begin{proof}
We follow the lines of Step $2$ of the proof of Theorem $2.10$ in \cite{mupo}. We want to show that 
\begin{equation}\label{strong}
\lim_{n\rightarrow\infty}\io\big(a(x,\nabla T_{k}(u_{n}))-a(x,\nabla T_{k}(u))\big)\cdot\nabla(T_{k}(u_{n})-T_{k}(u))=0
\end{equation}
in order to apply \cite[Lemma $5$]{bmp} and to conclude the proof.
\\In \eqref{ren_fin} we take $\varphi=(T_{k}(u_{n})-T_{k}(u))(1-\Psi_{\nu})$ and $S=\theta_r$, where $r>k$ and $\Psi_{\nu}$ is as in Lemma \ref{dalmaso}, obtaining
\begin{align}\label{1-4.1}
&\io a(x,\nabla T_k(u_{n}))\cdot\nabla(T_{k}(u_{n})-T_{k}(u))(1-\Psi_{\nu})\notag\\
&=-\int_{\{k<u_n<2r\}}a(x,\nabla u_{n})\cdot\nabla(T_{k}(u_{n})-T_{k}(u))\theta_r(u_n)(1-\Psi_{\nu}) \ \ \  (\rm{a}) \nonumber\\
&+\frac{1}{r}\int_{\{r<u_n<2r\}}a(x,\nabla u_{n})\cdot\nabla u_{n}(T_{k}(u_{n})-T_{k}(u))(1-\Psi_{\nu}) \ \ \ (\rm{b})\\
&+\io H(u_n)\theta_r(u_n)(T_{k}(u_{n})-T_{k}(u))(1-\Psi_{\nu})\mu_n \ \ \ (\rm{c})\nonumber\\
&+\io a(x,\nabla u_{n})\cdot\nabla\Psi_{\nu}(T_{k}(u_{n})-T_{k}(u))\theta_r(u_n).\notag\ \ \ (\rm{d})\nonumber
\end{align}
For $(\rm{a})$, we note that the term $\dis\{a(x,\nabla u_{n})\theta_r(u_n)\}$ is bounded in $L^{p'}(\Omega)^N$ with respect to $n$. Moreover we have that $\dis |\nabla T_k(u)|\chi_{\{u_n>k\}}$ converges to zero in $L^p(\Omega)$, which allows us to deduce that
\begin{equation}\label{A}
(\textrm{a})\le C\io|a(x,\nabla u_{n})\theta_r(u_n)\|\nabla T_{k}(u)|\chi_{\{u_n>k\}}=\epsilon(n).
\end{equation}
In the same way, we observe that $\dis\{a(x,\nabla u_{n})\cdot\nabla\Psi_{\nu}\theta_r(u_n)\}$ is bounded in $L^{p'}(\Omega)$ and that, by \eqref{weakconv}, $T_k(u_n)$ strongly converges to $T_k(u)$ in $L^p(\Omega)$, and so we arrive to
\begin{equation}\label{D}
(\textrm{d})\le\io |a(x,\nabla u_{n})\cdot\nabla\Psi_{\nu}\theta_r(u_n)\|(T_{k}(u_{n})-T_{k}(u))|=\epsilon(n).
\end{equation}
Now we focus on $(\rm{c})$, finding, by \eqref{approxmisura}, that
\begin{equation}\label{C.1}\begin{aligned}
(\textrm{c})&\leq\|H\|_{L^\infty(\RR)}\io|T_k(u_n)-T_k(u)|\mu_{n,d} \\ 
&+\io H(u_n)\theta_r(u_n)(T_k(u_n)-T_k(u))(1-\Psi_\nu)\mu_{n,c}.
\end{aligned}\end{equation}
Since $T_k(u_n)-T_k(u)$ is bounded in $W_0^{1,p}(\Omega) \cap L^\infty(\Omega)$ and converges to zero almost everywhere in $\Omega$, by Lemma \ref{Orsina}, the first term of the right hand side of \eqref{C.1} converges to zero as $n$ goes to infinity. As regards the second term we have that 
$$\io H(u_n)\theta_r(u_n)(T_k(u_n)-T_k(u))(1-\Psi_\nu)\mu_{n,c}\leq 2k\|H\|_{L^\infty(\RR)}\io (1-\Psi_\nu)\mu_{n,c},$$
which, through the narrow convergence of $\mu_{n,c}$ to $\mu_c$ and Lemma \ref{dalmaso}, implies
\begin{equation}\label{C}
(\textrm{c})\leq\epsilon(n,r,\nu).
\end{equation}
Gathering \eqref{A}, \eqref{D}, \eqref{C} in \eqref{1-4.1} we deduce
\begin{equation}\label{1-4.2}\begin{aligned}
&\io a(x,\nabla T_k(u_{n}))\cdot\nabla(T_{k}(u_{n})-T_{k}(u))(1-\Psi_{\nu})\\
&\le\epsilon(n,r,\nu)+\frac{2k}{r}\int_{\{r<u_n<2r\}}a(x,\nabla u_{n})\cdot\nabla u_{n}(1-\Psi_{\nu}).
\end{aligned}\end{equation}
Let us take $\varphi=\pi_r(u_n)(1-\Psi_{\nu})$ and $S=\theta_t$ in \eqref{ren_fin}, where $r,k,t\in\NN$, $r>k$, and $\pi_r(s)$ is given by \eqref{pi}. It results
\begin{equation}\label{Bq}\begin{aligned}
&\frac{1}{r}\int_{\{r<u_n<2r\}}a(x,\nabla u_{n})\cdot\nabla u_{n}\theta_t(u_n)(1-\Psi_{\nu})\\
&=\frac{1}{t}\int_{\{t<u_n<2t\}}a(x,\nabla u_{n})\cdot\nabla u_{n}\pi_r(u_n)(1-\Psi_{\nu})\ \ \  (\rm{a'}) \\
&+\io H(u_n)\pi_r(u_n)\theta_t(u_n)(1-\Psi_{\nu})\mu_n\ \ \  (\rm{b'})\\
&+\io a(x,\nabla u_{n})\cdot\nabla\Psi_{\nu}\pi_r(u_n)\theta_t(u_n).\ \ \  (\rm{c'})
\end{aligned}\end{equation}
As regards $(\rm{c'})$, thanks to Lebesgue Theorem, it results
\begin{equation*}\label{senzam}
\lim_{t\to\infty}\io a(x,\nabla u_{n})\cdot\nabla\Psi_\nu\pi_r(u_n)\theta_t(u_n)=\io a(x,\nabla u_{n})\cdot\nabla\Psi_\nu\pi_r(u_n).
\end{equation*}
Recalling that $\supp (\pi_r(s))=\{|s|\geq r\}$, that $u$ is almost everywhere finite and $|\nabla u_n|^{p-1}$ is bounded in $L^q(\Omega)$ for each $q<\frac{N}{N-1}$, then it follows from the H\"older inequality with exponents $q$ and $q'$, where $q<\frac{N}{N-1}$ is fixed, that
\begin{align*}
\left|\io a(x,\nabla u_{n})\cdot\nabla\Psi_\nu\pi_r(u_n)\right|&\leq\|\nabla\Psi_\nu\|_{L^\infty(\Omega)}\left(\io|\nabla u_n|^{(p-1)q}\right)^{\frac{1}{q}}\left|\{u_n\geq r\}\right|^\frac{1}{q'}\\
&\leq C\,\left|\{u_n\geq r\}\right|^\frac{1}{q'}=\epsilon(n,r),
\end{align*}
which implies 
\begin{equation}\label{Bq3}(\textrm{c}')\leq \epsilon(t,n,r).\end{equation}
As concerns $(\rm{b'})$ we have
\begin{equation}\label{Bq2}
\io H(u_n)\pi_r(u_n)\theta_t(u_n)(1-\Psi_{\nu})(\mu_{n,d}+\mu_{n,c})\le\|H\|_{L^{\infty}(\RR)}\io \pi_r(u_n)(1-\Psi_{\nu})(\mu_{n,d}+\mu_{n,c}).
\end{equation}
Finally we consider $(\textrm{a}')$. Letting $t$ go to infinity and recalling \eqref{r0}, we obtain
\begin{equation}\label{Bq1}\begin{aligned}
&\lim_{t\rightarrow\infty}\frac{1}{t}\int_{\{t<u_n<2t\}}a(x,\nabla u_{n})\cdot\nabla u_{n}\pi_r(u_n)(1-\Psi_{\nu})\\
&\le \lim_{t\rightarrow\infty}\frac{1}{t}\int_{\{t<u_n<2t\}}a(x,\nabla u_{n})\cdot\nabla u_{n}=0.
\end{aligned}\end{equation}
As $t$ goes to infinity in \eqref{Bq} and, by \eqref{Bq3}, \eqref{Bq2}, \eqref{Bq1}, we obtain
\begin{equation}\label{Bpenult}
\frac{1}{r}\int_{\{r<u_n<2r\}}a(x,\nabla u_{n})\cdot\nabla u_{n}(1-\Psi_{\nu})\le\epsilon(n,r)+\|H\|_{L^{\infty}(\RR)}\io \pi_r(u_n)(1-\Psi_{\nu})(\mu_{n,d}+\mu_{n,c}).\nonumber
\end{equation}
Since $\pi_r(u_n)$ converges to its almost everywhere limit weakly$^*$ in $L^{\infty}(\Omega)$ and weakly in $W^{1,p}_0(\Omega)$, we deduce, by Lemma \ref{Orsina}, that
$$\lim_{n\rightarrow\infty}\io \pi_r(u_n)(1-\Psi_{\nu})\mu_{n,d}=\io \pi_r(u)(1-\Psi_{\nu})d\mu_d.$$
As $u$ is $\operatorname{cap}_p$ almost everywhere finite, $\pi_r(u)$ converges to zero $\mu_d$-almost everywhere as $r\rightarrow\infty$; then, using Lebesgue Theorem for general measure, we obtain that 
$$\io \pi_r(u)(1-\Psi_{\nu})d\mu_d=\epsilon(r,\nu).$$
Moreover it follows from the narrow convergence of $\mu_{n,c}$ to $\mu_c$ and from Lemma \ref{dalmaso} that
$$\lim_{n\rightarrow\infty}\io \pi_r(u_n)(1-\Psi_{\nu})\mu_{n,c}\le\lim_{n\rightarrow\infty}\io(1-\Psi_{\nu})\mu_{n,c}=\io (1-\Psi_{\nu})d\mu_c\le C\nu.$$
Thus we obtain
\begin{equation}\label{con0}
\frac{1}{r}\int_{\{r<u_n<2r\}}a(x,\nabla u_{n})\cdot\nabla u_{n}(1-\Psi_{\nu})\le\epsilon(n,r,\nu),
\end{equation}
and then, going back to \eqref{1-4.2}, we conclude that
$$\io a(x,\nabla T_k(u_{n}))\cdot\nabla(T_{k}(u_{n})-T_{k}(u))(1-\Psi_{\nu})\le\epsilon(n,r,\nu).$$
Now we reason as follows
\begin{align}\label{tesiq}
&\io \big(a(x,\nabla T_k(u_{n}))-a(x,\nabla T_k(u))\big)\cdot\nabla(T_{k}(u_{n})-T_{k}(u))\nonumber\\
&=\io\big(a(x,\nabla T_k(u_{n}))-a(x,\nabla T_k(u))\big)\cdot\nabla(T_{k}(u_{n})-T_{k}(u))\Psi_{\nu}\nonumber\\
&+\io a(x,\nabla T_k(u_{n}))\cdot\nabla(T_{k}(u_{n})-T_{k}(u))(1-\Psi_{\nu})\\
&-\io a(x,\nabla T_k(u))\cdot\nabla(T_{k}(u_{n})-T_{k}(u))(1-\Psi_{\nu})\nonumber\\
&\le C\io\Big(|\nabla T_k(u_n)|^p+|\nabla T_k(u)|^p\Big)\Psi_{\nu}+\epsilon(n,r,\nu)\nonumber.
\end{align}
Now choosing as test function $(k-u_{n})^+\Psi_{\nu}$ in the weak formulation \eqref{pbapproxlimitata} we have
\begin{align*}
&-\io a(x,\nabla T_k(u_{n}))\cdot\nabla T_{k}(u_{n})\Psi_{\nu} + \io a(x,\nabla T_{k}(u_{n}))\cdot\nabla\Psi_{\nu} (k-u_{n})^+\\
&=\io H(u_{n})(k-u_{n})^+\Psi_{\nu}\mu_{n,d} + \io H(u_{n})(k-u_{n})^+\Psi_{\nu}\mu_{n,c},
\end{align*}
which implies, using $\mu_{n,d}\geq 0$ and \eqref{cara1},
\begin{align}\label{u1}
&\alpha\io |\nabla T_{k}(u_{n})|^p\Psi_\nu + \io H(u_{n})(k-u_{n})^+\Psi_{\nu}\mu_{n,c}\leq \io a(x,\nabla T_{k}(u_{n}))\cdot\nabla\Psi_{\nu} (k-u_{n})^+.
\end{align}
Moreover, since $T_{k}(u_n)$ is bounded in $\dis W^{1,p}_0(\Omega)$, it follows by an application of the H\"older inequality and by Lemma \ref{dalmaso} that 
\begin{equation}\label{u2}
\io a(x,\nabla T_{k}(u_{n}))\cdot\nabla\Psi_{\nu} (k-u_{n})^+\leq k\|T_k(u_{n})\|_{W_0^{1,p}(\Omega)}\|\Psi_{\nu}\|_{W_0^{1,p}(\Omega)}\leq \epsilon(n,\nu).
\end{equation}
By \eqref{u1} and \eqref{u2} we obtain
\begin{equation} \label{u3}
\io |\nabla T_{k}(u_{n})|^p\Psi_\nu=\epsilon(n,\nu)
\end{equation}
and
\begin{equation}\label{u4}
\io H(u_{n})(k-u_{n})^+\Psi_{\nu}\mu_{n,c}=\epsilon(n,\nu).
\end{equation}
Finally, by \eqref{tesiq} and \eqref{u3}, we have
$$\io \big(a(x,\nabla T_k(u_{n}))-a(x,\nabla T_k(u))\big)\cdot\nabla(T_{k}(u_{n})-T_{k}(u))\leq \epsilon(n,r,\nu),$$
which is \eqref{strong} as desired. In conclusion it holds	
\begin{equation*}\label{strongconv}
T_k(u_n) \rightarrow T_k(u) \text{ strongly in } W^{1,p}_0(\Omega) \text{ for every fixed } k>0,
\end{equation*}
yielding also that $\nabla u_n$ converges almost everywhere in $\Omega$ to $\nabla u$.
\end{proof}
\begin{remark}\label{remarkconvergenze}
It follows from Lemma \ref{lemmastime} and Lemma \ref{stronghlim} that, if $p>2-\frac{1}{N}$, $u_n$ converges to $u$ strongly in $W_{0}^{1,q}(\Omega)$ for every $q<\frac{N(p-1)}{N-1}$. Otherwise, if $1<p\leq2-\frac{1}{N}$, $u_n^{p-1}$ converges to $u^{p-1}$ strongly in $L^q(\Omega)$ for every $q<\frac{N}{N-p}$ and $|\nabla u_n|^{p-1}$ converges to $|\nabla u|^{p-1}$ strongly in $L^q(\Omega)$ for every $q<\frac{N}{N-1}$. In all cases we have	
\begin{equation}\label{conva}
a(x,\nabla u_{n}) \rightarrow a(x,\nabla u) \text{ strongly in } L^{q}(\Omega)^N \text{ for every } q<\frac{N}{N-1}.
\end{equation}
\end{remark}
Now we are ready to prove Theorem \ref{teoexrinuniqueness} in case $\gamma=0$, namely when $H(0)<\infty$.
\begin{proof}[Proof of Theorem \ref{teoexrinuniqueness} in case $\gamma=0$]
In order to prove the existence part of the theorem we only need to show that $u$, almost everywhere limit of the solutions $u_n$ to \eqref{pbapproxlimitata}, is a renormalized solution to \eqref{pbmain}. Indeed we already know, by Lemma \ref{lemmastime}, that $T_k(u)\in W^{1,p}_0(\Omega)$. If $S \in W^{1,\infty}(\mathbb{R})$ with $\supp(S)\subset [-M,M]$ and $\varphi \in W^{1,p}_0(\Omega)\cap L^\infty(\Omega)$, taking $S(u_n)\varphi$ as test function in the weak formulation of \eqref{pbapproxlimitata} we obtain
\begin{align}\label{dif0}
\io a(x,\nabla u_{n})\cdot\nabla\varphi S(u_n) + \io a(x,\nabla u_{n})\cdot \nabla u_{n}S'(u_n)\varphi=\io H(u_{n})S(u_{n})\varphi\mu_n.
\end{align}	
It follows from Lemma \ref{stronghlim} that we have 
\begin{align*}
\lim_{n\to\infty} \io a(x,\nabla u_{n})\cdot \nabla u_{n}S'(u_n)\varphi&=\lim_{n\to\infty}\io a(x,\nabla T_{M}(u_{n}))\cdot \nabla T_{M}(u_{n})S'(T_{M}(u_n))\varphi \\
&=\io a(x,\nabla T_{M}(u))\cdot \nabla T_{M}(u)S'(T_{M}(u))\varphi \\
&= \io a(x,\nabla u)\cdot \nabla u S'(u)\varphi,
\end{align*}
and
\begin{align*}
\lim_{n\to\infty}\io a(x,\nabla u_{n})\cdot\nabla\varphi S(u_n)&=\lim_{n\to\infty}\io a(x,\nabla T_{M}(u_{n}))\cdot\nabla \varphi S(T_{M}(u_{n})) \\
&=\io a(x,\nabla T_{M}(u))\cdot\nabla \varphi S(T_{M}(u))\\
&=\io a(x,\nabla u)\cdot \nabla\varphi S(u). 
\end{align*}
Hence, in order to deduce \eqref{ren1}, we need to pass to the limit the right hand side of \eqref{dif0}. We split it as follows
\begin{align}\label{dif3}
\io H(u_{n})S(u_{n})\varphi\mu_{n}=\io H(u_{n})S(u_{n})\varphi\mu_{n,d}+\io H(u_{n})S(u_{n})\varphi\mu_{n,c},
\end{align}
treating the two terms in the right hand side of the previous separately.
\\Let $H_{j}(s)$ be a sequence of functions in $C^1(\RR^+)$ such that 
$$H_{j}'\in L^\infty(\RR^+)\cap L^{1}(\RR^+),\quad\dis \|H_j-H\|_{L^\infty(\RR^+)}\leq \frac{1}{j}.$$
Since $u$ is cap$_p$-quasi continuous, $H$, $H_j$ and $S$ are continuous and finite functions on $\RR$, then $H_j(u)S(u)\varphi$ and $H(u)S(u)\varphi$ are $\mu_d$-measurable. Then we have
\begin{align}\label{hj}
&\left|\io H(u_{n})S(u_{n})\varphi\mu_{n,d}-\io H(u)S(u)\varphi d\mu_d\right|\leq \left|\io (H(u_{n})-H_j(u_{n}))S(u_{n})\varphi\mu_{n,d}\right|\notag \\
&+\left|\io (H_j(u)-H(u))S(u)\varphi d\mu_d\right|+\left|\io H_j(u_{n})S(u_{n})\varphi\mu_{n,d}-H_{j}(u)S(u)\varphi d\mu_{d}\right|  \\
&\leq \frac{C}{j}+\left|\io H_j(u_{n})S(u_{n})\varphi\mu_{n,d}-H_{j}(u)S(u)\varphi d\mu_{d}\right|.\nonumber 
\end{align}
Now, thanks to the assumptions on the functions $H_j,S$ and $\varphi$ and to \eqref{Tkunif}, it is easy to verify that $H_j(u_{n})S(u_{n})\varphi$ is bounded in $W_{0}^{1,p}(\Omega)\cap L^\infty(\Omega)$ with respect to $n\in\NN$ and its almost everywhere limit is given by $H_j(u)S(u)\varphi$. Then, by Lemma \ref{Orsina} and \eqref{approxmisura}, we get
\begin{equation*}
\lim_{n\to\infty}\io H_j(u_{n})S(u_{n})\varphi\mu_{n,d}=\io H_j(u)S(u)\varphi d\mu_d.
\end{equation*}
Now, using the Lebesgue Theorem for general measure and the assumptions on the sequence $H_j$, we are able to pass to the limit also with respect to $j$, concluding that
\begin{equation*}
\lim_{j\to\infty}\lim_{n\to\infty}\io H_j(u_{n})S(u_{n})\varphi\mu_{n,d}=\io H(u)S(u)\varphi d\mu_d
\end{equation*}
and that $H(u)S(u)\varphi \in L^1(\Omega,\mu_d)$. As regards the second term in the right hand side of \eqref{dif3}, we first observe that, since $S$ has compact support, there exist $k > 0$ and $c_k > 0$ such that $S(s)\le c_k(k-s)^+$ for every $s\in \RR$. Then we have
\begin{align*}
&\io H(u_{n})S(u_{n})\varphi\mu_{n,c}=\io H(u_{n})S(u_{n})\varphi\Psi_\nu\mu_{n,c}+\io H(u_{n})S(u_{n})\varphi (1-\Psi_\nu)\mu_{n,c}\\
&\leq \|\varphi\|_{L^\infty(\Omega)}c_k\io H(u_{n})(k-u_{n})^+\Psi_\nu\mu_{n,c}+\|H\|_{L^\infty(\RR)}\|\varphi\|_{L^\infty(\Omega)}\|S\|_{L^\infty(\RR)}\io (1-\Psi_\nu)\mu_{n,c}.
\end{align*}
So, by Lemma \ref{dalmaso} and \eqref{u4}, letting first $n$ go to infinity and then $\nu$ go to zero, we obtain
\begin{equation*}\label{dif5}
\lim_{n\to\infty}\io H(u_{n})S(u_{n})\varphi\mu_{n,c}=0,
\end{equation*} 
which proves \eqref{ren1}, as desired. 
\\Now we want to prove that \eqref{ren2} holds true. 
\\First we need to prove that $u$ is a distributional solution of \eqref{pbmain}. If $\varphi\in C^1_c(\Omega)$, we have
\begin{equation}\label{lim}
\io a(x,\nabla u_{n})\cdot \nabla\varphi=\io H(u_{n})\varphi\mu_{n,d}+\io H(u_{n})\varphi\mu_{n,c}.
\end{equation}
For the left hand side of the previous, by \eqref{conva} we deduce
$$ \lim_{n\to\infty} \io a(x,\nabla u_{n})\cdot \nabla\varphi=\io a(x,\nabla u)\cdot\nabla\varphi.$$
Concerning the first term on the right hand side of \eqref{lim}, we reason as in \eqref{hj} yielding
\begin{align*}
\left|\io H(u_n)\varphi\mu_{n,d}-\io H(u)\varphi d\mu_d\right|\leq &\left|\io (H(u_n)-H_j(u_n))\varphi\mu_{n,d}\right|\\
&+\left|\io (H_j(u)-H(u))\varphi d\mu_d\right|\nonumber\\
&+\left|\io H_j(u_n)\varphi\mu_{n,d}-H_j(u)\varphi d\mu_{d}\right|\nonumber\\
\leq &\frac{C}{j}+\left|\io H_j(u_n)\varphi\mu_{n,d}-H_j(u)\varphi d\mu_{d}\right|.\nonumber
\end{align*}
To prove that the last term in the previous goes to zero with respect to $n$, it is sufficient to show that $H_j(u_n)\varphi$ is bounded with respect to $n$, with $j$ fixed, in $W^{1,p}_0(\Omega)\cap L^\infty(\Omega)$.
Clearly $H_j(u_n)\varphi$ is bounded, with respect to $n$, in $L^\infty(\Omega)$. To show the boundedness in $W^{1,p}_0(\Omega)$ of $H_j(u_n)\varphi$, we take $\theta_k(u_n)\int_0^{T_{2k}(u_n)}\left|H'_j(s)\right|ds$ as test function in the weak formulation of \eqref{pbapproxlimitata}. Then we find 
\begin{align*}
\io a(x,\nabla u_n)\cdot\nabla T_{2k}(u_n)&\left|H'_j(T_{2k}(u_n))\right|\theta_k(u_n) =\io H(u_n)\left(\theta_k(u_n)\int_0^{T_{2k}(u_n)}\left|H'_j(s)\right|ds\right)\mu_n\\
&+\frac{1}{k}\int_{\{k<u_n<2k\}}a(x,\nabla u_n)\cdot\nabla u_n\left(\int_0^{T_{2k}(u_n)}\left|H'_j(s)\right|ds\right)\\
&\le \|H\|_{L^\infty(\mathbb{R})}\|H_j\|_{L^\infty(\mathbb{R})}\|\mu_n\|_{L^1(\Omega)} +\epsilon(k)\\
&\leq C+\epsilon(k),
\end{align*}
since $H'_j\in L^1(\RR^+)$ and \eqref{r0} holds. Then, by \eqref{cara1}, we deduce
$$\io |\nabla T_{2k}(u_n)|^p\left|H'_j(T_{2k}(u_n))\right|\theta_k(u_n)\leq C+\epsilon(k)$$
namely
$$\io |\nabla u_n|^p\left|H'_j(u_n)\right|\theta_k(u_n)\leq C+\epsilon(k).$$
Letting $k\to\infty$ in the previous and using Fatou Lemma, we find
$$\frac{1}{\|H'_j\|^{p-1}_{L^\infty(\mathbb{R})}}\io\left|\nabla H_j(u_n)\right|^p\leq\io|\nabla u_n|^p\left|H'_j(u_n)\right|\leq C,$$
which implies that $H_j(u_n)\varphi$ is bounded in $W^{1,p}_0(\Omega)$ with respect to $n$.
\\Now we go back to the second term on the right hand side of \eqref{lim}. By \eqref{approxmisura}, recalling that $\varphi \in C^1_c(\Omega)$, it results
\begin{align}\label{dis3}
\left|\io H(u_{n})\varphi\mu_{n,c}-\io H(\infty)\varphi d\mu_{c}\right| \leq &\left|\io H(u_{n})\varphi\mu_{n,c} -\io H(\infty)\varphi\mu_{n,c}\right| \\
&+ \left|\io H(\infty)\varphi\mu_{n,c}-\io H(\infty)\varphi d\mu_{c}\right| \nonumber \\
\leq &\|\varphi\|_{L^{\infty}(\Omega)}\io |H(u_{n})-H(\infty)|\mu_{n,c} + \epsilon(n). \nonumber
\end{align}
By \eqref{h}, for every $\eta>0$ there exist $s_{\eta}>0$ and $L_{\eta}>0$ such that
\begin{equation}
\label{H1}
|H(s)-H(\infty)|\leq\eta, \qquad \forall s> s_{\eta}
\end{equation}
and, using that $H(s)>0$ for $s\ge 0$, we have
\begin{equation}\label{H2}
|H(s)-H(\infty)| \leq H(s)L_{\eta}(2s_{\eta}-s)^+, \qquad \forall s\in[0,s_{\eta}].
\end{equation}
It follows from \eqref{H1}, \eqref{H2}, \eqref{approxmisura} and applying \eqref{u4} with $k=2s_{\eta}$ that
\begin{align*}
\io |H(u_{n})-H(\infty)|\mu_{n,c}=&\io |H(u_{n})-H(\infty)|\Psi_\nu\mu_{n,c}+\io |H(u_{n})-H(\infty)|(1-\Psi_\nu)\mu_{n,c}\\
\leq &\eta\int_{\{u_{n}> s_{\eta}\}}\Psi_{\nu}\mu_{n,c}+L_{\eta}\int_{\{u_n\leq s_{\eta}\}} H(u_{n})(2s_{\eta}-u_{n})^+\Psi_\nu \mu_{n,c}\\
&+2\|H\|_{L^{\infty}(\RR)}\io (1-\Psi_\nu)\mu_{n,c}\\
\le&\epsilon(n,\nu,\eta).
\end{align*}
Hence, by \eqref{dis3}, we have
\begin{align*}
\left|\io H(u_{n})\varphi\mu_{n,c}-\io H(\infty)\varphi d\mu_{c}\right|\leq \epsilon(n,\nu,\eta),
\end{align*}
which implies that
\begin{equation}\label{genultima2.14mp}
\lim_{n\to\infty}\io H(u_{n})\varphi\mu_{n,c}=H(\infty)\io\varphi d\mu_{c},
\end{equation}
then \eqref{distrdef} is proved. 
\\Now taking $S=\theta_t$ and $\varphi\in C^1_c(\Omega)$ in \eqref{ren1} we obtain
\begin{align*}
\frac{1}{t}\int_{\{t<u<2t\}}a(x,\nabla u)\cdot\nabla u \varphi=-\io H(u)\theta_{t}(u)\varphi d\mu_{d}+\io a(x,\nabla u)\cdot\nabla\varphi \theta_{t}(u).
\end{align*}
Now, using that $\theta_t$ belongs to $C_b(\RR)$ and that $u$ is cap$_p$-almost everywhere defined, by Lebesgue's Theorem for general measure we obtain
\begin{align*}
\lim_{t\to\infty}\frac{1}{t}\int_{\{t<u<2t\}}a(x,\nabla u)\cdot\nabla u \varphi=-\io H(u)\varphi d\mu_{d}+\io a(x,\nabla u)\cdot\nabla\varphi,
\end{align*}
which implies, by \eqref{distrdef}, that
\begin{equation}\label{con1}
\lim_{t\to\infty}\frac{1}{t}\int_{\{t<u<2t\}}a(x,\nabla u)\cdot\nabla u \varphi=H(\infty)\io \varphi d\mu_{c} \qquad \forall \varphi\in C^1_c(\Omega).
\end{equation}
By the density of $C^1_c(\Omega)$ in $C_c(\Omega)$, \eqref{con1} is true when $\varphi\in C_c(\Omega)$. 
\\Now, if $\varphi\in C_b(\Omega)$, we have $\varphi\Psi_\nu\in C_c(\Omega)$ and then
\begin{equation}\label{con2}
\lim_{t\to\infty}\frac{1}{t}\int_{\{t<u<2t\}}a(x,\nabla u)\cdot\nabla u \Psi_\nu\varphi=H(\infty)\io \varphi\Psi_\nu d\mu_{c}\qquad \forall \varphi\in C_b(\Omega).
\end{equation}
Applying \eqref{con0} with $r=t$, and letting $n$ go to infinity, we find
\begin{equation}\label{con3}
\lim_{t\to\infty}\frac{1}{t}\int_{\{t<u<2t\}}a(x,\nabla u)\cdot\nabla u(1-\Psi_\nu)\varphi=\epsilon(\nu)\qquad \forall \varphi\in C_b(\Omega).
\end{equation}
Then, by \eqref{con2} and \eqref{con3}, we deduce
\begin{equation*}\label{con4}
\lim_{t\to\infty}\frac{1}{t}\int_{\{t<u<2t\}}a(x,\nabla u)\cdot\nabla u \varphi=H(\infty)\io \varphi\Psi_\nu d\mu_{c}+\epsilon(\nu)\qquad \forall \varphi\in C_b(\Omega).
\end{equation*}
Letting $\nu$ go to zero, by Lemma \ref{dalmaso}, we obtain \eqref{ren2}.
\\If we further ask that $H$ is non-increasing and that $\mu_c\equiv0$, uniqueness follows as in Theorem $2.11$ of \cite{mupo}. This concludes the proof of Theorem \ref{teoexrinuniqueness} if $\gamma=0$.
\end{proof}

\section{The approximation scheme if $H$ is singular}\label{sec:maggiore0}
	
In this section we collect some properties of the solutions to the scheme of approximation which will be the basis to prove Theorems \ref{teoexrinuniqueness}, \ref{teoexistence} in case $\gamma>0$, namely when the function $H$ can blow up at the origin.
\\We will find a solution to the problem  passing to the limit in the following approximation
\begin{equation}\begin{cases}
\displaystyle-\operatorname{div}(a(x,\nabla u_{n,m})) = H_n(u_{n,m})(\mu_d + \mu_{m}) &  \text{in}\, \Omega,\\
u_{n,m}=0 & \text{on}\ \partial \Omega.
\label{pbapproxeps}
\end{cases}\end{equation}
where $H_n=T_n(H)$ and $\mu_{m}$ is, once again, a sequence of nonnegative functions in $L^{\infty}(\Omega)$, bounded in $L^1(\Omega)$, that converges to $\mu_c$ in the narrow topology of measures. We recall that $H$ satisfies \eqref{h} and \eqref{h1} with $\gamma>0$ and that $a$ is a Carath\'eodory function such that \eqref{cara1}, \eqref{cara2} and \eqref{cara3} with $1<p<N$ hold true.
\\ The existence of a nonnegative renormalized solution $u_{n,m}$ to problem \eqref{pbapproxeps} is guaranteed by the result proven in Section \ref{exgamma=0}. Moreover it follows from Lemma \ref{equivrindis} that $u_{n,m}$ is also a distributional solution to \eqref{pbapproxeps}.
\\ \\For the sake of simplicity, since until the passage to the limit it will be not necessary to distinguish between $n$ and $m$, we will consider the following approximation in place of \eqref{pbapproxeps}
\begin{equation}\begin{cases}
\displaystyle -\operatorname{div}(a(x,\nabla u_{n})) = H_n(u_{n})(\mu_d + \mu_n) &  \text{in}\, \Omega, \\
u_{n}=0 & \text{on}\ \partial \Omega.
\label{pbapprox}
\end{cases}
\end{equation}
The first step is proving the local uniform positivity for $u_{n}$, which will assure that the possibly singular right hand side is locally integrable with respect to $\mu_d$.
\begin{lemma}
Let $u_{n}$ be a solution to \eqref{pbapprox}. Then
\begin{equation} 
\forall\;\omega\subset\subset\Omega\; \ \ \exists\; c_{\omega}>0: u_{n}\geq c_{\omega} \ \ \text{cap$_p$-a.e.}\;\ \text{in}\ \;\omega,\;\ \ \forall n\ge n_0,
\label{posunif}
\end{equation}
for some $n_0>0$.
\label{locpos}
\end{lemma}
\begin{proof} 
The proof is similar to the one of Lemma $3.4$ in \cite{do} given for $p=2$. For this reason we just sketch it. For some $n_0\in \mathbb{N}$, it is possible to construct a non-increasing function $h\in C_b(\RR)$ such that $h(s)\leq H_n(s)$ for every $n\ge n_0$ and for all $s\geq 0$. 
\\Then we can consider the following problem
\begin{equation}
\begin{cases}
\displaystyle -\operatorname{div}(a(x,\nabla v)) = h(v)\mu_d &  \text{in}\, \Omega, \\
v=0 & \text{on}\ \partial \Omega,
\label{pbv}
\end{cases}
\end{equation}
for which the existence of a nonnegative renormalized solution $v\not\equiv0$ follows once again from Section \ref{exgamma=0}. It can be proven that there exists $\overline{r}>0$ such that $\mu_d\lfloor_{\{v< \overline{r}\}}\not\equiv0$ and that $h(v)\mu_d$ is a diffuse measure respect to $p$-capacity. Then, from Definition $2.29$ and Remark $2.32$ in \cite{dmop}, we deduce that $T_{\overline{r}}(v)\in W^{1,p}_0(\Omega)$ solves the following
\begin{equation*}
\displaystyle -\operatorname{div}(a(x,\nabla T_{\overline{r}}(v))) =h(v)\mu_d\lfloor_{\{v<\overline{r}\}} + \lambda_{\overline{r}} \ge 0 \ \  \text{in}\ \Omega,
\end{equation*}
where $\lambda_{\overline{r}}$ is a nonnegative diffuse measure concentrated on the set $\{v=\overline{r}\}$. Hence we can apply the strong maximum principle (see, for instance, Theorem $1.2$ of \cite{t}), obtaining
\begin{equation*} 
\forall\;\omega\subset\subset\Omega \ \ \ \exists\; C_{\omega,{\overline{r}}}>0: v\ge T_{\overline{r}}(v)\geq c_{\omega}:= C_{\omega,{\overline{r}}}>0\; \ \text{a.e. in}\;\ \omega\,.
\end{equation*}
Now we consider the renormalized formulations of \eqref{pbapprox} and of \eqref{pbv}, taking $S=\theta_k$ in both equations and $\varphi=\theta_k(v)T_r(v-u_{n})^+$ in \eqref{pbapprox}, $\varphi=\theta_k(u_{n})T_r(v-u_{n})^+$ in \eqref{pbv}, where $r>0$ is fixed. 
\\We have
\begin{align}\label{diffren1}
&\io\big(a(x,\nabla v)-a(x,\nabla u_{n})\big)\cdot\nabla T_r(v-u_{n})^+\theta_k(v)\theta_k(u_{n}) \\
&=\frac{1}{k}\int_{\{k<u_{n}<2k\}}a(x,\nabla v)\cdot \nabla u_{n}T_r(v-u_{n})^+\theta_k(v)-\frac{1}{k}\int_{\{k<v<2k\}}a(x,\nabla u_{n})\cdot\nabla vT_r(v-u_{n})^+\theta_k(u_{n})\nonumber\\
&+\frac{1}{k}\int_{\{k<v<2k\}}a(x,\nabla v)\cdot\nabla vT_r(v-u_{n})^+\theta_k(u_{n})-\frac{1}{k}\int_{\{k<u_{n}<2k\}}a(x,\nabla u_{n})\cdot\nabla u_{n}T_r(v-u_{n})^+\theta_k(v)\nonumber\\
&+\io \big(h(v)-H_n(u_{n})\big)T_r(v-u_{n})^+\theta_k(v)\theta_k(u_{n})d\mu_d-\io H_n(u_{n})T_r(v-u_{n})^+\theta_k(v)\theta_k(u_{n})\mu_{n}.\nonumber
\end{align}
Since the concentrated part of the datum is zero both in \eqref{pbapprox} and in \eqref{pbv}, from the definition of renormalized solution we obtain that the third and the fourth term of the right hand side of \eqref{diffren1} go to zero as $k$ goes to infinity. With the same argument, after an application of the H\"{o}lder inequality, we deduce that the first and the second term of the right hand side of the previous go to zero as $k$ goes to infinity. Since the last term of \eqref{diffren1} is nonpositive and $h$ is non-increasing, we deduce that
\begin{align*}
&\io\big(a(x,\nabla v)-a(x,\nabla u_{n})\big)\cdot \nabla T_r(v-u_{n})^+\theta_k(v)\theta_k(u_{n})\\
\le&\int_{\{v\ge u_{n}\}}\big(h(v)-H_n(u_{n})\big)T_r(v-u_{n})^+\theta_k(v)\theta_k(u_{n})d\mu_d\nonumber\\
\le&\int_{\{v\ge u_{n}\}}\big(h(u_{n})-H_n(u_{n})\big)T_r(v-u_{n})^+\theta_k(v)\theta_k(u_{n})d\mu_d.\nonumber
\end{align*}
Since $h\le H_n$ for every $n\ge n_0$, $h$ and $H_n$ are continuous and $u_{n}$ is cap$_{p}$-almost everywhere defined, we have $\big(h(u_{n})-H_n(u_{n})\big)\le 0 $ cap$_{p}$-almost everywhere in $\Omega$ if $n\ge n_0$. Moreover, applying in the previous the Fatou Lemma first in $k$ and then in $r$, we deduce
$$\io\big(a(x,\nabla v)-a(x,\nabla u_{n})\big)\cdot \nabla(v-u_{n})\chi_{\{v\ge u_{n}\}}\le0,$$
which, by \eqref{cara3}, implies 
$$\chi_{\{v\ge u_{n}\}}\equiv0 \ \ \text{if $n\geq n_0$}.$$ 
Hence we have proved that \eqref{posunif} holds almost everywhere in $\Omega$.
\\Now, if $\omega\subset\subset\Omega$ and $k_{\omega}>c_{\omega}$, then 
\begin{equation}\label{komega}
T_{k_{\omega}}(u_n)\ge c_{\omega}\;\text{a.e. in $\omega$.}
\end{equation}
Using the definition of the set of Lebesgue points of a function $f$ applied with the choice $f=\left.T_{k_{\omega}}(u_n)\right|_{\omega}$ and Lebesgue differentiation Theorem, we deduce that
\begin{equation*}
T_{k_{\omega}}(u_n)\ge c_{\omega}\;\text{in}\;\mathcal{L}_{\left.T_{k_{\omega}}(u_n)\right|_{\omega}}.
\end{equation*}
Since $T_{k_{\omega}}(u_n)\in W^{1,p}(\omega)$, using Proposition $8.6$ of \cite{p} we obtain that $\operatorname{cap}_p(\omega\setminus\mathcal{L}_{\left.T_{k_{\omega}}(u_n)\right|_{\omega}})=0$. In particular \eqref{komega} holds cap$_p$-almost everywhere on $\omega$ and, since $u_n\ge T_{k_{\omega}}(u_n)$, we conclude that \eqref{posunif} holds cap$_p$-almost everywhere in $\omega$.
\end{proof}
Now we are interested in providing some a priori estimates up to the boundary in order to give a weak sense to the Dirichlet datum.
\begin{lemma}
Let $u_n$ be a solution to \eqref{pbapprox}. Then $T_k^\frac{\tau-1+p}{p}(u_n)$ is bounded in $W^{1,p}_0(\Omega)$  for every fixed $k>0$ where $\tau = \max(1,\gamma)$.
\label{bordo}
\end{lemma} 
\begin{proof}
 We take as test functions in the renormalized formulation of \eqref{pbapprox} $S=\theta_r$ and $\varphi=T_k^\tau(u_n)$
where $r>k$. We let $r\to\infty$ and use that the concentrated part of the datum in \eqref{pbapprox} is zero. Then we obtain the following
\begin{equation}\label{g<=1}
\begin{aligned}
\displaystyle \int_{\Omega} |\nabla T_k^{\frac{\tau-1+p}{p}}(u_{n})|^p &\le Cs_0^{\tau-\gamma}\int_{\{u_n<s_0\}}(d\mu_d + \mu_n)+Ck^\tau\|H\|_{L^\infty ([s_0,+\infty))}\int_{\{u_n\ge s_0\}}(d\mu_d+\mu_n)\\ 
&\le C(k^\tau+1),
\end{aligned}
\end{equation}
as desired.
\end{proof}
\begin{remark}\label{gamma>1}
Let us underline that, in case $\gamma>1$, $T_k(u_n)$ is bounded in $W^{1,p}_{loc}(\Omega)$ with respect to $n\in\NN$ for $n$ large enough and for every fixed $k>0$ . Indeed, it follows from Lemma \ref{locpos} and Lemma \ref{bordo} that for every $\omega \subset \subset \Omega$ it results
\begin{align*}
&\left(\frac{\gamma+p-1}{p}\right)^p c_{\omega}^{\gamma-1}\int_{\omega}|\nabla T_k(u_n)|^p=\left(\frac{\gamma+p-1}{p}\right)^p\io T_k(u_n)^{\gamma-1}|\nabla T_k(u_n)|^p \\
&=\io|\nabla T_k(u_n)^{\frac{\gamma+p-1}{p}}|^p\le C(1+k^\gamma).
\end{align*}
\end{remark}
We prove local a priori estimates for $u_n$.
\begin{lemma}\label{lemmastimehsing}
Let $u_n$ be a solution to \eqref{pbapprox}. Then:
\begin{itemize}
\item[i)] if $p>2-\frac{1}{N}$, $u_n$ is bounded in $W^{1,q}_{loc}(\Omega)$ for every $q<\frac{N(p-1)}{N-1}$;
\item[ii)] if $1<p\le 2-\frac{1}{N}$, $u_n^{p-1}$ is bounded in $L^q_{loc}(\Omega)$ for every $q<\frac{N}{N-p}$ and $|\nabla u_n|^{p-1}$ is bounded in $L^q_{loc}(\Omega)$ for every $q<\frac{N}{N-1}$.
\end{itemize}
Moreover there exists an almost everywhere finite function $u$ such that $u_n$ converges almost everywhere to $u$ in $\Omega$, $u$ is locally cap$_p$-almost everywhere finite, locally cap$_p$-quasi continuous and such that
\begin{equation}
\forall\;\omega\subset\subset\Omega\; \ \ \exists\; c_{\omega}>0: u\geq c_{\omega}\;\text{cap$_p$-a.e. in}\;\omega,\label{posunifu}
\end{equation}
$$H(u)\in L^{\infty}(\omega;\mu_d)\quad\forall \omega \subset \subset \Omega.$$
\end{lemma}
\begin{proof}
By Lemma \ref{bordo} and Remark \ref{gamma>1}, we have that $T_k(u_n)$ is bounded  in $W^{1,p}_{loc}(\Omega)$ with respect to $n\in\NN$ for each $k>0$ fixed and for all $\gamma>0$. Then, localizing the proof Lemma \ref{lemmastime}, we deduce immediately that i) and ii) hold true and that there exists an almost everywhere finite function $u$ such that $u_n$ converges almost everywhere to $u$ in $\Omega$. Moreover, using \eqref{g<=1}, once again Remark \ref{gamma>1} and localizing Lemma \ref{dalmaso2}, we obtain that $u$ is locally cap$_p$-almost everywhere finite and locally cap$_p$-quasi continuous.
Now, letting $n\rightarrow\infty$ in \eqref{posunif}, we deduce that
\begin{equation} 
\forall\;\omega\subset\subset\Omega\; \ \ \exists\; c_{\omega}>0: u\geq c_{\omega} \ \ a.e. \;\ \text{in}\ \;\omega,\label{posunifuqo}
\end{equation}
and, since $T_k(u)\in W^{1,p}_{loc}(\Omega)$, we can proceed as at the end of the proof of Lemma \ref{locpos} to conclude that \eqref{posunifuqo} holds also cap$_p$-almost everywhere in $\omega$, that is \eqref{posunifu}. Using \eqref{posunifu} and the fact that $H(s)$ is finite if $s>0$, we deduce $H(u)\in L^{\infty}(\omega;\mu_d)$ for every $\omega \subset \subset \Omega$.
\end{proof}
\begin{remark}\label{remark<1}
Recalling Lemma \ref {bordo}, in the case $\gamma\le1$ we can improve the previous Lemma obtaining that i) and ii) hold true globally in $\Omega$ and that $u$ is cap$_p$-almost everywhere finite and cap$_p$-quasi continuous.
\end{remark}
The next Lemma is a strong convergence result for the truncations, this time (compare with Lemma \ref{stronghlim}, see also \cite{do} for $p=2$) in the local space $W^{1,p}_{loc}(\Omega)$.
\begin{lemma}\label{stronghsing}
Let $u_n$ be a solution to \eqref{pbapprox}. Then $T_k(u_n)$ converges to $T_k(u)$ in $W^{1,p}_{loc}(\Omega)$ for every $k>0$.
\end{lemma}
\begin{proof}
The proof is similar to the one of Lemma \ref{stronghlim}. It suffices to take $\varphi=(T_{k}(u_{n})-T_{k}(u))(1-\Psi_{\nu})\psi$ and $S=\theta_r$ ($r>k$) in the renormalized formulation of \eqref{pbapprox} where $\psi\in C_c^1(\Omega)$
 such that for $\omega \subset \subset \Omega$ we have
\begin{eqnarray*}
\begin{cases}
0\leq\psi\leq1\;\text{on}\;\Omega,\\
\psi\equiv1\;\text{on}\;\omega\subset\subset\Omega.
\end{cases}
\end{eqnarray*}
Hence, through the local estimates and proceeding in an analogous way as to prove the strong convergence of truncations in Lemma \ref{stronghlim}, we obtain 
\begin{equation*}
\lim_{n\rightarrow\infty}\io\big(a(x,\nabla T_{k}(u_{n}))-a(x,\nabla T_{k}(u))\big)\cdot\nabla(T_{k}(u_{n})-T_{k}(u))\psi=0,
\end{equation*}
so that, by \cite[Lemma $5$]{bmp}, we have that $T_k(u_n)$ converges to $T_k(u)$ strongly in $W^{1,p}_{loc}(\Omega)$ for every $k>0$ and $\nabla u_n$  converges to $\nabla u$ almost everywhere in $\Omega$. This concludes the proof.
\end{proof}
\begin{remark}
Analogously to Remark \ref{remarkconvergenze}, from Lemma \ref{lemmastimehsing} and Lemma \ref{stronghsing} we deduce that if $p>2-\frac{1}{N}$ then $u_n$ converges to $u$ strongly in $W_{loc}^{1,q}(\Omega)$ for every $q<\frac{N(p-1)}{N-1}$. Otherwise if $1<p\leq2-\frac{1}{N}$ then $u_n^{p-1}$ converges to $u^{p-1}$ strongly in $L^q_{loc}(\Omega)$ for every $q<\frac{N}{N-p}$ and	$|\nabla u_n|^{p-1}$ converges to $|\nabla u|^{p-1}$ strongly in $L^q_{loc}(\Omega)$ for every $q<\frac{N}{N-1}$. In all cases we have	
 \begin{equation}\label{passaggiolimitea}
 a(x,\nabla u_{n}) \rightarrow a(x,\nabla u) \text{ strongly in } L^{q}_{loc}(\Omega)^N \text{ for every } q<\frac{N}{N-1}.
 \end{equation}
\end{remark}

\section{Proof of the existence and uniqueness results}\label{mainresults}

In this section we first prove Theorem \ref{teoexistence}, and then Theorem \ref{teoexrinuniqueness} in full generality, namely for $\gamma>0$.
\\Indeed, in order to prove Theorem \ref{teoexrinuniqueness}, we need that the scheme of approximation actually takes to a distributional solution to \eqref{pbmain}, which is the content of Theorem \ref{teoexistence}.
\begin{proof}[Proof of Theorem \ref{teoexistence}]
Let $u_{n,m}$ be a renormalized solution to \eqref{pbapproxeps}. We need to prove that its almost everywhere limit $u$, whose existence is guaranteed by Lemma \ref{lemmastimehsing}, is a distributional solution to \eqref{pbmain}. 
\\It follows from Lemma \ref{bordo} that \eqref{troncate} holds. Hence we just need to show \eqref{distrdef}, namely we have to pass to the limit first in $m$ and then in $n$ the following weak formulation
\begin{equation}
\io a(x,\nabla u_{n,m})\cdot\nabla\varphi=\io H_n(u_{n,m})\varphi d\mu_d+\io H_n(u_{n,m})\varphi\mu_m,\;\ \ \ \forall\varphi\in C^1_c(\Omega).
\label{ndist}
\end{equation}
Thanks to \eqref{passaggiolimitea}, we are able to pass to the limit the first term on left hand side of the previous as $n,m\rightarrow\infty$. Now we pass to the right hand side of \eqref{ndist}. For $n\in \mathbb{N}$ fixed and proceeding as to deduce \eqref{genultima2.14mp}, we find that
$$\lim_{m\rightarrow\infty}\io H_n(u_{n,m})\varphi\mu_m=H_n(\infty)\io\varphi d\mu_c,$$
and, since for $n\in\NN$ large enough it results $H_n(\infty)=H(\infty)$, we get
$$\lim_{n\rightarrow\infty}\lim_{m\rightarrow\infty}\io H_n(u_{n,m})\varphi\mu_m=H(\infty)\io\varphi d\mu_c.$$
For the first term on the right hand side of \eqref{ndist} we observe that, by Lemma \ref{stronghsing}, it yields that $T_k(u_{n,m})$ strongly converges to $T_k(u)$ in $W^{1,p}_{loc}(\Omega)$. This implies (see Lemma $3.5$ of \cite{kkm}) that $T_k(u_{n,m})$ converges to $T_k(u)$ cap$_{p}$-almost everywhere in $\omega$ for each $k>0$ fixed and for $\omega \subset \subset \Omega$. Being $u_{n,m}$ and $u$ cap$_{p}$-almost everywhere finite functions, we deduce that $u_{n, m}$ converges cap$_{p}$-almost everywhere to $u$ in $\omega$ for each $\omega\subset\subset\Omega$. Hence $H_n(u_{n,m})$ converges to $H(u)$ cap$_{p}$-almost everywhere in $\supp (\varphi)$. Thus we are in position to apply the Lebesgue Theorem for general measure since 
$$|H_n(u_{n,m}) \varphi |  \le \|H\|_{L^{\infty}([c_{\text{supp} (\varphi)},\infty))} \|\varphi\|_{L^\infty(\Omega)} \in L^1(\Omega,\mu_d),$$
where we have used that, by Lemma \ref{locpos}, $u_{n,m}\geq c_{\supp(\varphi)}$ cap$_p$-almost everywhere on $\supp(\varphi)$ for $n$ and $m$ large enough. Hence we have proved that it results
$$\lim_{n,m\to\infty}\io H_n(u_{n,m})\varphi d\mu_d=\io H(u)\varphi d\mu_d,$$
and then $u$ is a distributional solution to \eqref{pbmain}. This concludes the proof.
\end{proof}
\begin{proof}[Proof of Theorem \ref{teoexrinuniqueness} in case $\gamma>0$]
Let $u_{n,m}$ be a renormalized solution to \eqref{pbapproxeps}, then it follows from the proof of Theorem \ref{teoexistence} that its almost everywhere limit $u$ is a distributional solution to \eqref{pbmain}. We have that $u_{n,m}$ is such that
\begin{align}\label{new1}
&\io a(x,\nabla u_{n,m})\cdot\nabla\varphi S(u_{n,m}) + \io a(x,\nabla u_{n,m})\cdot \nabla u_{n,m}S'(u_{n,m})\varphi\\
&=\io H_n(u_{n,m})S(u_{n,m})\varphi d\mu_d+\io H_n(u_{n,m})S(u_{n,m})\varphi\mu_m,\nonumber
\end{align}	
where $S\in W^{1,\infty}(\mathbb{R})$ with supp$(S)$ $\subset [-M,M]$ and $\varphi\in C^1_c(\Omega)$.
\\As regards the left hand side of \eqref{new1}, since, by Lemma \ref{stronghsing}, $T_M(u_{n,m})$ strongly converges to $T_M(u)$ in $W^{1,p}_{loc}(\Omega)$, by \eqref{cara2} and Vitali's Theorem, we obtain
\begin{align*}
\lim_{n\to\infty}\lim_{m\to\infty}&\left(\io a(x,\nabla u_{n,m})\cdot\nabla\varphi S(u_{n,m})+\io a(x,\nabla u_{n,m})\cdot \nabla u_{n,m}S'(u_{n,m})\varphi\right)\\
&=\io a(x,\nabla u)\cdot \nabla\varphi S(u)+\io a(x,\nabla u)\cdot \nabla u S'(u)\varphi.
\end{align*}
For the first term on the right hand side of \eqref{new1} we observe that, using once again Lemma \ref{locpos}, it results
$$H_n(u_{n,m}) S(u_{n,m})\varphi \le \|H\|_{L^{\infty}([c_{\text{supp} (\varphi)},\infty))} \|\varphi\|_{L^\infty(\Omega)} \|S\|_{L^\infty(\mathbb{R})}\in L^1(\Omega,\mu_d).$$
Then, thanks to the cap$_p$-almost everywhere convergence of $u_{n,m}$ to $u$, we can apply the Lebesgue Theorem for general measure, obtaining
$$\lim_{n\to\infty}\lim_{m\to\infty}\io H_n(u_{n,m})S(u_{n,m})\varphi d\mu_d=\io H(u)S(u)\varphi d\mu_d.$$
For the second term on the right hand side of \eqref{new1} we have, proceeding as in the proof of Theorem 3.4 in the case $\gamma=0$, that there exist $k>0$ and $c_k>0$ such that $S(s)\leq c_k(k-s)^+$ for every $s\in\RR$ and 
\begin{equation}\label{app1}
\begin{aligned}
\io H_n(u_{n,m})S(u_{n,m})\varphi\mu_m &\leq c_k\|\varphi\|_{L^{\infty}(\Omega)}\io H_n(u_{n,m})(k-u_{n,m})^+\Psi_\nu\mu_m \\
&+\|H\|_{L^{\infty}([c_{\text{supp}(\varphi)},\infty))}\|S\|_{L^{\infty}(\mathbb{R})}\|\varphi\|_{L^{\infty}(\Omega)}\io(1-\Psi_\nu)\mu_m. 
\end{aligned}
\end{equation}
Using $S(s)=(k-|s|)^+$ and $\varphi=\Psi_\nu$ in the renormalized formulation of \eqref{pbapproxeps} and dropping positive terms we obtain
\begin{equation}\label{app2}
\begin{aligned}
\io H_n(u_{n,m})(k-u_{n,m})^+\Psi_\nu\mu_m &\leq \io a(x,\nabla T_k(u_{n,m}))\cdot \nabla\Psi_\nu (k-u_{n,m})^+ \\
&\leq k \|T_k(u_{n,m})\|_{W^{1,p}(\operatorname{supp}(\Psi_\nu))} \|\Psi_\nu\|_{\sob}. 
\end{aligned}
\end{equation}
Then, from \eqref{app1} and \eqref{app2}, we deduce, applying Lemma \ref{dalmaso}, Lemma \ref{bordo}, Remark \ref{gamma>1} and letting $\nu\to0$, that
$$\lim_{n\to\infty}\lim_{m\to\infty}\io H_n(u_{n,m})S(u_{n,m})\varphi\mu_m=0.$$
Hence we have proved
\begin{align}
\label{new2}
&\io a(x,\nabla u)\cdot\nabla\varphi S(u) + \io a(x,\nabla u)\cdot \nabla u S'(u)\varphi=\io H(u)S(u)\varphi d\mu_d,
\end{align}
for every $S \in W^{1,\infty}(\mathbb{R})$ with compact support and for every $\varphi\in C^1_c(\Omega)$, namely \eqref{ren1} for a smaller class of test functions $\varphi\in C^1_c(\Omega)$. Note that \eqref{new2} holds true also if $\gamma>1$.
\\Now we take $S=\theta_t$ in \eqref{new2} and we obtain
\begin{align*}
\frac{1}{t}\int_{\{t< u< 2t\}}a(x,\nabla u)\cdot\nabla u \varphi=-\io H(u)\theta_{t}(u)\varphi d\mu_{d}+\io a(x,\nabla u)\cdot\nabla\varphi \theta_{t}(u).
\end{align*}
We pass to the limit in $t$ obtaining
\begin{align*}
\lim_{t\to\infty}\frac{1}{t}\int_{\{t< u< 2t\}}a(x,\nabla u)\cdot\nabla u \varphi=-\io H(u)\varphi d\mu_{d}+\io a(x,\nabla u)\cdot\nabla\varphi,
\end{align*}
which implies, since $u$ is a distributional solution to \eqref{pbmain}, that
\begin{equation}\label{new3}
\lim_{t\to\infty}\frac{1}{t}\int_{\{t< u< 2t\}}a(x,\nabla u)\cdot\nabla u \varphi=H(\infty)\io \varphi d\mu_{c} \qquad \forall \varphi\in C^1_c(\Omega).
\end{equation}
By the density of $C^1_c(\Omega)$ in $C_c(\Omega)$, \eqref{new3} is true when $\varphi\in C_c(\Omega)$. Now, if $\varphi\in C_b(\Omega)$, we have $\varphi\Psi_\nu\in C_c(\Omega)$ and then
\begin{equation}\label{new4}
\lim_{t\to\infty}\frac{1}{t}\int_{\{t< u< 2t\}}a(x,\nabla u)\cdot\nabla u \Psi_\nu\varphi=H(\infty)\io \varphi\Psi_\nu d\mu_{c}\qquad \forall \varphi\in C_b(\Omega).
\end{equation}
We want to prove that
\begin{equation}\label{new5}
\lim_{t\to\infty}\frac{1}{t}\int_{\{t< u< 2t\}}a(x,\nabla u)\cdot\nabla u(1-\Psi_\nu)\varphi=\epsilon(\nu)\qquad \forall \varphi\in C_b(\Omega).
\end{equation}
Choosing in the renormalized formulation of \eqref{pbapproxeps} $\varphi=\pi_t(u_{n,m})(1-\Psi_{\nu})$ and $S=\theta_r$, with $t>1$, we obtain 
\begin{align}\label{newabcd}
&\frac{1}{t}\int_{\{t<u_{n,m}<2t\}}a(x,\nabla u_{n,m})\cdot\nabla u_{n,m}\theta_r(u_{n,m})(1-\Psi_{\nu})\nonumber\\
=&\frac{1}{r}\int_{\{r<u_{n,m}<2r\}}a(x,\nabla u_{n,m})\cdot\nabla u_{n,m}\pi_t(u_{n,m})(1-\Psi_{\nu}) \ \ \ (a)\nonumber \\
&+\io H_n(u_{n,m})\pi_t(u_{n,m})\theta_r(u_{n,m})(1-\Psi_{\nu})d\mu_d \ \ \ (b) \\
&+\io H_n(u_{n,m})\pi_t(u_{n,m})\theta_r(u_{n,m})(1-\Psi_{\nu})\mu_m \ \ \ (c)\nonumber \\
&+\io a(x,\nabla u_{n,m})\cdot\nabla\Psi_{\nu}\pi_t(u_{n,m})\theta_r(u_{n,m}). \ \ \ (d) \nonumber
\end{align}
As concerns (d), thanks to the Lebesgue Theorem, we deduce 
\begin{equation*}
\lim_{r\to\infty}\io a(x,\nabla u_{n,m})\cdot\nabla\Psi_\nu\pi_t(u_{n,m})\theta_r(u_{n,m})=\io a(x,\nabla u_{n,m})\cdot\nabla\Psi_\nu\pi_t(u_{n,m}).
\end{equation*}
Recalling that $u$ is almost everywhere finite, that $|\nabla u_{n,m}|^{p-1}$ is bounded in $L^q(\omega)$ for each $q<\frac{N}{N-1}$ where $\omega:=$supp$(\Psi_\nu)$, using \eqref{cara2} and H\"older inequality with exponents $q$ and $q'$, with $1<q<\frac{N}{N-1}$ fixed, we find
\begin{align*}
\left|\io a(x,\nabla u_{n,m})\cdot\nabla\Psi_\nu\pi_t(u_{n,m})\right|&\leq\|\nabla\Psi_\nu\|_{L^\infty(\Omega)}\left(\int_\omega|\nabla u_{n,m}|^{(p-1)q}\right)^{\frac{1}{q}}\left|\{x\in \omega : u_{n,m}(x)\geq t\}\right|^\frac{1}{q'}\\
&\leq C\,\left|\{x\in \omega : u_{n,m}(x)\geq t\}\right|^\frac{1}{q'}=\epsilon(m,n,t).
\end{align*}
Then
\begin{equation}\label{newd}
(d) \le \epsilon(m,n,t).
\end{equation}
Concerning (b) and (c), once again by Lebesgue Theorem, we deduce that
\begin{align}
\label{newb}
\io H_n(u_{n,m})\pi_t(u_{n,m})\theta_r(u_{n,m})(1-\Psi_{\nu})d\mu_{d} &\le\|H\|_{L^{\infty}([1,+\infty))}\io \pi_t(u_{n,m})(1-\Psi_{\nu})d\mu_d \\
&=\epsilon(m,n,t), \nonumber
\end{align}
and that
$$\lim_{r\to\infty}\io H_n(u_{n,m})\pi_t(u_{n,m})\theta_r(u_{n,m})(1-\Psi_{\nu})\mu_m=\io H_n(u_{n,m})\pi_t(u_{n,m})(1-\Psi_{\nu})\mu_m.$$
By the narrow convergence of $\mu_{m}$ and Lemma \ref{dalmaso}, we obtain
\begin{equation}\label{newc}
\io H_n(u_{n,m})\pi_t(u_{n,m})(1-\Psi_{\nu})\mu_{m}\le \|H\|_{L^{\infty}([1,+\infty))}\io(1-\Psi_{\nu})\mu_{m}=\epsilon(m,\nu)
\end{equation}
Finally, by \eqref{r0}, we obtain
\begin{equation}\begin{aligned}\label{newa}
&\frac{1}{r}\int_{\{r<u_{n,m}<2r\}}a(x,\nabla u_{n,m})\cdot\nabla u_{n,m}\pi_t(u_{n,m})(1-\Psi_{\nu})\\
&\le \frac{1}{r}\int_{\{r<u_{n,m}<2r\}}a(x,\nabla u_{n,m})\cdot\nabla u_{n,m}=\epsilon(r).
\end{aligned}	
\end{equation}
Letting $r$ go to infinity in \eqref{newabcd} and using \eqref{newd}, \eqref{newb},\eqref{newc} and \eqref{newa}, we get
\begin{equation*}
\frac{1}{t}\int_{\{t<u_{n,m}<2t\}}a(x,\nabla u_{n,m})\cdot\nabla u_{n,m}(1-\Psi_{\nu})=\epsilon(m,n,t,\nu).
\end{equation*}
Then, by Vitali's Theorem, letting $m$, $n$ and $t$ go to infinity we deduce \eqref{new5}. As a consequence of \eqref{new4} and \eqref{new5}, letting $\nu$ go to zero, by Lemma \ref{dalmaso} we have
\begin{equation}\label{ener}
\lim_{t\to\infty}\frac{1}{t}\int_{\{t< u< 2t\}}a(x,\nabla u)\cdot\nabla u \varphi=H(\infty)\io \varphi d\mu_{c},
\end{equation}
for all $\varphi\in C_b(\Omega)$. Hence \eqref{ren2} holds and, in order to deduce that $u$ is a renormalized solution, we just need to show that \eqref{new2} holds for a larger class of test functions, namely for $\varphi \in W^{1,p}_0(\Omega)\cap L^\infty(\Omega)$. 
\\It follows from Remark \ref{remark<1} that $T_k(u) \in W^{1,p}_0(\Omega),$ for every $k>0$.
Now let $\phi_n \in C^1_c(\Omega)$ be a sequence of nonnegative functions that converges in $W^{1,p}_0(\Omega)$ to a nonnegative $v\in W^{1,p}_0(\Omega)\cap L^\infty(\Omega)$ and let $\rho_\eta$ be a smooth mollifier. We take $\varphi=\rho_\eta \ast (v \wedge \phi_n)\in C^1_c(\Omega)$ in \eqref{new2} where $v \wedge \phi_n:= \inf (v,\phi_n)$, obtaining
\begin{equation}\label{ren1c1c_2}\begin{aligned}
\int_{\Omega}& a(x,\nabla u)\cdot\nabla (\rho_\eta \ast (v \wedge \phi_n)) S(u) + \int_{\Omega}a(x,\nabla u)\cdot\nabla u S'(u)\rho_\eta \ast (v \wedge \phi_n)\\
 = & \int_{\Omega}H(u)S(u)(\rho_\eta \ast (v \wedge \phi_n)) d\mu_d.
\end{aligned}\end{equation}
We assume that $\supp(S) \subset [-M,M]$ and we analyze the three terms in \eqref{ren1c1c_2} separately. 
\\As concerns the first term on the left hand side of \eqref{ren1c1c_2}, using that 
$$a(x,\nabla u)S(u)=a(x,\nabla T_M(u))S(T_M(u))\in L^{p'}(\Omega)^N,$$ 
that $\rho_\eta \ast (v \wedge \phi_n)$ strongly converges to $v \wedge \phi_n$ in $W^{1,p}_0(\Omega)$ as $\eta\to0$ and that $v \wedge \phi_n$ strongly converges to $v$ in $W^{1,p}_0(\Omega)$ as $n\to\infty$, we deduce
\begin{equation}\label{I}
\int_{\Omega} a(x,\nabla u)\cdot\nabla (\rho_\eta \ast (v \wedge \phi_n)) S(u)=\int_{\Omega} a(x,\nabla u)\cdot\nabla vS(u)+\epsilon(\eta,n).
\end{equation} 
We consider now the second term on the left hand side of \eqref{ren1c1c_2}. Since 
$$a(x,\nabla u)\cdot \nabla u S'(u)=a(x,\nabla T_M(u))\cdot\nabla T_M(u) S'(T_M(u))\in L^1(\Omega)$$ 
and $\rho_\eta \ast (v \wedge \phi_n)$ converges to $v$ weakly* in $L^\infty(\Omega)$ as $\eta\to0$ and $n\to\infty$, we have that 
\begin{equation}\label{II}
\int_{\Omega}a(x,\nabla u)\cdot\nabla u S'(u)\rho_\eta \ast (v \wedge \phi_n)=\int_{\Omega} a(x,\nabla u)\cdot\nabla uS'(u)v+\epsilon(\eta,n).
\end{equation} 
\\Finally we consider the right hand side of \eqref{ren1c1c_2}. Since $\rho_\eta \ast (v \wedge \phi_n)$ converges to $v\wedge\phi_n$ cap$_p$-almost everywhere as $\eta\to0$ and the following inequality holds true cap$_p$-almost everywhere
$$H(u)S(u)(\rho_\eta \ast (v \wedge \phi_n))\leq \|H\|_{L^\infty([c_{\supp(\phi_n)},\infty))}\|S\|_{L^\infty(\RR)}\|v \wedge \phi_n\|_{L^\infty(\Omega)}\in L^1(\Omega,\mu_d)\;$$
by Lebesgue's Theorem for general measure we find
\begin{equation}\label{III}
\int_{\Omega}H(u)S(u)(\rho_\eta \ast (v \wedge \phi_n)) d\mu_d=\int_{\Omega}H(u)S(u)(v \wedge \phi_n) d\mu_d+\epsilon(\eta).
\end{equation}
Hence, putting together \eqref{I}, \eqref{II} and \eqref{III}, we find
\begin{equation}\label{IV}
\int_{\Omega} a(x,\nabla u)\cdot\nabla vS(u)+\int_{\Omega} a(x,\nabla u)\cdot\nabla uS'(u)v=\int_{\Omega}H(u)S(u)(v \wedge \phi_n) d\mu_d+\epsilon(\eta,n)
\end{equation}
Now, since we can write $S$ as $S^+-S^-$, where $S^+$ and $S^-$ are the positive and the negative part of $S$, we can assume, without loss of generality, that $S\geq0$. In particular, $H(u)S(u)(v \wedge \phi_n)$ is a sequence of nonnegative and $\mu_d$-measurable functions (recall that $\phi_n$ has compact support for each $n\in\NN$) that converges cap$_p$-almost everywhere to $H(u)S(u)v$. Hence we can apply Fatou's Lemma in \eqref{IV} obtaining 
\begin{align*}
&\io H(u)S(u)v d\mu_d\leq\liminf_{n\to\infty}\int_{\Omega}H(u)S(u)(v \wedge \phi_n) d\mu_d\\
&=\int_{\Omega} a(x,\nabla u)\cdot\nabla vS(u)+\int_{\Omega} a(x,\nabla u)\cdot\nabla uS'(u)v+\epsilon(\eta,n). 
\end{align*}
The latter one implies that 
$$H(u)S(u)v\in L^1(\Omega,\mu_d)\quad\forall v\in W^{1,p}_0(\Omega)\cap L^\infty(\Omega)\;\text{s.t.}\;v\geq0.$$
Then, since 
$$H(u)S(u)(v\wedge \phi_n)\underset{n\to\infty}{\longrightarrow}H(u)S(u)v\quad\mu_d\text{-a.e.}$$
and 
$$H(u)S(u)(v\wedge \phi_n)\leq H(u)S(u)v\quad\mu_d\text{-a.e.}$$
by Lebesgue's Theorem we deduce that
$$\lim_{n\to\infty}\int_{\Omega}H(u)S(u)(v\wedge \phi_n) d\mu_d=\int_{\Omega}H(u)S(u)vd\mu_d.$$
In conclusion, passing to the limit first as $\eta\to0$ and then as $n\to\infty$ in \eqref{IV}, we obtain 
\begin{equation}\label{ren3c1c}
\int_{\Omega} a(x,\nabla u)\cdot\nabla vS(u) + \int_{\Omega}a(x,\nabla u)\nabla u S'(u)v = \int_{\Omega}H(u)S(u)v d\mu_d
\end{equation}
for every $S \in W^{1,\infty}(\mathbb{R})$ with compact support and for every nonnegative $v\in W^{1,p}_0(\Omega)\cap L^\infty(\Omega)$. Since it is possible to write each $v\in W^{1,p}_0(\Omega)\cap L^\infty(\Omega)$ as the difference between its positive and its negative part (as done before for the test function $S$), we trivially deduce that \eqref{ren3c1c} holds for all $v\in W^{1,p}_0(\Omega)\cap L^\infty(\Omega)$. Hence, recalling also \eqref{ener}, we conclude that $u$ is a renormalized solution to \eqref{pbmain}.
\\Once again, if $H$ is non-increasing and $\mu_c\equiv0$, uniqueness easily follows adapting  the proof of Theorem $2.11$ of \cite{mupo} to the case of a general $p>1$. This concludes the proof.
\end{proof}

\section{Some remarks when $H$ degenerates}\label{H=0}

It is worth to analyze more in depth what kind of phenomena could appear in case of a nonnegative function $H$, namely if we remove the request of strict positivity for $H$. 
\\We recall that the problem is given by
\begin{equation}
\begin{cases}
\displaystyle -\operatorname{div}(a(x,\nabla u)) = H(u)\mu &  \text{in}\, \Omega, \\
u=0 & \text{on}\ \partial \Omega.
\label{pbfinal}
\end{cases}
\end{equation}
Here we assume that $\mu$ is a nonnegative bounded Radon measure on $\Omega$ such that $\mu_c\equiv 0$ and that the function $a$ satisfies \eqref{cara1}, \eqref{cara2} and \eqref{cara3}. Concerning the function $H:(0,+\infty)\to [0,+\infty)$, we will assume that is continuous, such that \eqref{h} and \eqref{h1} hold and that it is zero for some $s>0$.  
\\ We will prove that, under these assumptions on the lower order term, there exists a solution to \eqref{pbfinal} that is bounded  and that belongs, at least locally, to the energy space. 
\\This kind of remark was already given in \cite{dcgop} for more regular data. We state the results and give just a brief idea of the proofs. 
\begin{theorem}\label{teoexrinuniquenessfinal}
Let us assume that $\mu_c\equiv0$ and that $0\leq\gamma\leq1$. If $s_1>0$ is the smallest positive value such that $H(s_1)=0$, then there exists a renormalized solution $u$ to \eqref{pbfinal} with $u\in W^{1,p}_0(\Omega)\cap L^\infty(\Omega)$ and $\|u\|_{L^\infty(\Omega)}\le s_1$.
\end{theorem}
\begin{theorem}\label{teoexistencefinal}
Let us assume that $\mu_c\equiv0$. If $s_1>0$ is the smallest positive value such that $H(s_1)=0$, then there exists a distributional solution $u$ to \eqref{pbfinal} with $u\in W^{1,p}_{loc}(\Omega)\cap L^\infty(\Omega)$ and $\|u\|_{L^\infty(\Omega)}\le s_1$. 
\end{theorem} 
Our first observation is that the assumption $H(s)>0$ for all $s\ge 0$ is used in the proof of Theorems \ref{teoexrinuniqueness} and \ref{teoexistence} only to show that the solution blows up on the support of $\mu_c$ (see \eqref{H2}).
\\Hence, if $\mu_c\equiv 0$, the proofs of Theorems \ref{teoexrinuniqueness} and \ref{teoexistence} remain valid even if $H$ is just nonnegative and, in order to prove Theorems \ref{teoexrinuniquenessfinal} and \ref{teoexistencefinal}, we only need to show the improvement in the regularity of the solution. 
\\Precisely, we will show that, under these assumptions on the lower order term, the schemes of approximation \eqref{pbapproxlimitata} and \eqref{pbapprox} (i.e. the approximations that led us to the existence results, respectively, in case $\gamma=0$ and $\gamma>0$), admit a sequence of solutions that is, respectively, bounded in $W^{1,p}_0(\Omega) \cap L^\infty(\Omega)$ if $\gamma\le 1$ and in $W^{1,p}_{loc}(\Omega) \cap L^\infty(\Omega)$ if $\gamma> 1$.
\\ \\ We recall that the scheme of approximation \eqref{pbapproxlimitata}, used in the case $\gamma=0$, is given by
\begin{equation}\begin{cases}
\displaystyle -\operatorname{div}(a(x,\nabla u_n)) = H(u_n)\mu_n &  \text{in}\, \Omega, \\
u_n=0 & \text{on}\ \partial \Omega,
\label{pbapproxlimitatafinal}
\end{cases}\end{equation}
where $H$ is bounded and $\mu_n=\mu_{n,d}\in L^\infty(\Omega)$ is bounded in $L^1(\Omega)$ and such that \eqref{approxmisura} holds.
\\We define on $[0,+\infty)$ the continuous function $\h$ as follows
\begin{equation}\label{H*}
\h(s)=\begin{cases}
H(s) \quad&\text{if}\;s<s_1,\\
0\quad&\text{if}\;s\geq s_1,
\end{cases}
\end{equation}
and we consider the following problem
\begin{equation}\begin{cases}
\displaystyle -\operatorname{div}(a(x,\nabla u^*_n)) =\h(u^*_n)\mu_{n} &  \text{in}\, \Omega, \\
u^*_n=0 & \text{on}\ \partial \Omega.
\label{pbapprox*1}
\end{cases}
\end{equation}
The latter problem has a weak solution $u^*_n\in W^{1,p}_0(\Omega)$, that is also nonnegative.
Now taking $G_{s_1}(u^*_n)$ as test function in \eqref{pbapprox*1}, we immediately find
$$\io|\nabla G_{s_1}(u^*_n)|^p=0$$ 
which implies $u^*_n\leq s_1$ almost everywhere in $\Omega$. Hence, recalling \eqref{H*}, we conclude that $u^*_n$ solves also \eqref{pbapproxlimitatafinal}. Moreover, having in mind the $L^\infty$-estimate for $u^*_n$ and taking $u^*_n$ itself as test function in the weak formulation of \eqref{pbapproxlimitatafinal}, we deduce that $u^*_n$ is bounded in $W^{1,p}_0(\Omega)$. This is sufficient to deduce Theorem \ref{teoexrinuniquenessfinal} if $\gamma=0$.
\\ \\The scheme of approximation introduced to prove Theorems \ref{teoexrinuniqueness} and \ref{teoexistence}
in the case $\gamma> 0$ is instead given by
\begin{equation}\begin{cases}
\displaystyle -\operatorname{div}(a(x,\nabla u_{n})) = H_n(u_{n})\mu_d  &  \text{in}\, \Omega, \\
u_{n}=0 & \text{on}\ \partial \Omega,
\label{pbapproxfinal}
\end{cases}
\end{equation}
where $H_n= T_n(H)$. In this case we consider the following problem
\begin{equation}\begin{cases}
\displaystyle -\operatorname{div}(a(x,\nabla u^*_n)) =\h_n(u^*_n)\mu_d &  \text{in}\, \Omega, \\
u^*_n=0 & \text{on}\ \partial \Omega,
\label{pbapprox*}
\end{cases}
\end{equation}
with $\h_n(s)=T_n(\h(s))$ for each $n\in\NN$. Applying Theorem \ref{teoexrinuniquenessfinal} in the case $\gamma=0$, we deduce that, if $n\in \mathbb{N}$ is fixed, there exists a renormalized solution $u^*_n\in W^{1,p}_0(\Omega)\cap L^\infty(\Omega)$ to \eqref{pbapprox*}. 
\\To prove the positivity of the sequence $u^*_n$, proceeding as done to deduce \eqref{posunif}, it is sufficient to construct on $[0,+\infty)$ a nonnegative function $h$ that is not identically zero, non-increasing, continuous, bounded and such that 
$$h(s)\leq \h_n(s)\;\text{for all $s>0$ and for $n$ large enough.}$$
Since $H(s)$ is continuous for each $s>0$ and $s_1$, with $s_1>s_0>0$, is the smallest zero of $H$, there exists $s^*\in[0,s_0]$ such that
$$H(s^*)=\min_{[0,s_0]}H(s)>0.$$
A good candidate for $h$ is then the following function
$$h(s)=\begin{cases}
H(s^*)\quad&\text{if}\;0\leq s< s^*,\\
\dis\frac{H(s^*)}{(s_0-s^*)}(s_0-s)\quad&\text{if}\;s^*\leq s\leq s_0,\\
0\quad&\text{if}\;s> s_0.
\end{cases}$$
From this point onwards, we can proceed as in Lemma \ref{locpos} to prove that 
$$\forall\;\omega\subset\subset\Omega\; \ \ \exists\; c_{\omega}>0: u^*_n\geq c_{\omega} \ \ \text{cap$_p$-a.e.} \;\ \text{in}\ \;\omega\;\text{for $n$ large enough}.$$Since, once again taking $G_{s_1}(u_n^*)$, it is possible to prove that $u^*_n\leq s_1$ almost everywhere in $\Omega$, the function $u^*_n$ turns out to be a solution to \eqref{pbapproxfinal}. 
\\Now we take as test function in the renormalized formulation of \eqref{pbapproxfinal} the following ones
$$\begin{cases}
S=\theta_r,\;\varphi=u^*_n\;&\text{if}\;\gamma\leq1,\\
S=\theta_r,\;\varphi=(u^*_n)^\gamma\;&\text{if}\;\gamma>1,
\end{cases}$$
where $r>0$.
\\ In case $\gamma\leq1$, as $r\to\infty$ we find
\begin{align*}
\alpha\int_{\Omega}|\nabla u^*_n|^p \le \int_{\Omega}a(x,\nabla u^*_n) \cdot  \nabla u^*_n &= \int_{\{u^*_n< s_0\}} H_n(u^*_n)u^*_nd\mu_d+\int_{\{u^*_n\geq s_0\}} H_n(u^*_n)u^*_nd\mu_d\\
&\leq\left(Cs_0^{1-\gamma}+\|H\|_{L^\infty([s_0,s_1))}s_1\right)\|\mu_d\|_{\mathcal{M}(\Omega)},
\end{align*}
namely that $u^*_n$ is bounded in $W^{1,p}_0(\Omega)$.
\\If $\gamma>1$, we find instead
\begin{align*}
\alpha\gamma c_\omega^{\gamma-1} \int_{\omega}|\nabla u^*_n|^p &\le \gamma\int_{\Omega}a(x,\nabla u^*_n) \cdot  \nabla u^*_n(u^*_n)^{\gamma-1}\\
&\leq\left(C+\|H\|_{L^\infty([s_0,s_1))}s_1^\gamma\right)\|\mu_d\|_{\mathcal{M}(\Omega)},
\end{align*}
i.e. that $u^*_n$ is bounded in $W^{1,p}_{loc}(\Omega)$. From now on, we can proceed as in the proof of Theorems \ref{teoexrinuniqueness} and \ref{teoexistence} in order to obtain Theorem \ref{teoexrinuniquenessfinal} for $\gamma>0$ and Theorem \ref{teoexistencefinal}.

\end{document}